\documentclass[10pt,reqno]{amsart}
\usepackage{amsmath,amssymb,latexsym,esint,cite,mathrsfs}
\usepackage{verbatim,cite,relsize,color,soul}
\usepackage{color,enumitem,graphicx}
\usepackage[colorlinks=true,urlcolor=blue, citecolor=red,linkcolor=blue,
linktocpage,pdfpagelabels, bookmarksnumbered,bookmarksopen]{hyperref}
\usepackage[hyperpageref]{backref}

\usepackage{tikz}
\usepackage[font=small,skip=5pt]{caption}
\usepackage{enumitem}

\usepackage{amsrefs}

\usepackage{lmodern}
 
\usepackage[T1]{fontenc}

\usepackage{geometry}
\numberwithin{equation}{section}

\newtheorem{theorem}{Theorem}[section]
\newtheorem{lemma}[theorem]{Lemma}
\newtheorem{proposition}[theorem]{Proposition}

\newtheoremstyle{remarkstyle}
{}{}{}{}{\bfseries}{.}{ }{\thmname{#1}\thmnumber{ #2}\thmnote{ (#3)}}
\theoremstyle{remarkstyle}
\newtheorem{remark}{Remark}[section]

\newcommand{\N}{\mathbb N}
\newcommand{\Z}{\mathbb Z}

\newcommand{\R}{\mathbb R}
\newcommand{\C}{\mathbb C}

\newcommand{\Vc}{\mathcal V}
\newcommand{\Gc}{\mathcal G}

\newcommand{\Hc}{\mathcal H}

\newcommand{\Hs}{\mathscr H}

\newcommand{\Mcal}{\mathcal M}

\newcommand{\cdotb}{\boldsymbol{\cdot}}
\newcommand{\vareps}{\varepsilon}
\DeclareMathOperator*{\loc}{loc}

\DeclareMathOperator*{\rea}{Re}

\title[NLS with critical rotational speed]
{Existence and stability of standing waves for nonlinear Schr\"odinger equations with a critical rotational speed}

\author[V. D. Dinh]{Van Duong Dinh}
\address[V. D. Dinh]{Ecole Normale Sup\'erieure de Lyon \& CNRS, UMPA (UMR 5669), France
	and 
	Department of Mathematics, Ho Chi Minh City University of Education, 280 An Duong Vuong, Ho Chi Minh City, Vietnam}
\email{contact@duongdinh.com}


\begin{document}
	
	\begin{abstract}
		We study the existence and stability of standing waves associated to the Cauchy problem for the nonlinear Schr\"odinger equation (NLS) with a critical rotational speed and an axially symmetric harmonic potential. This equation arises as an effective model describing the attractive Bose-Einstein condensation in a magnetic trap rotating with an angular velocity. By viewing the equation as NLS with a constant magnetic field and with (or without) a partial harmonic confinement, we establish the existence and orbital stability of prescribed mass standing waves for the equation with mass-subcritical, mass-critical, and mass-supercritical nonlinearities. Our result extends a recent work of [Bellazzini-Boussa\"id-Jeanjean-Visciglia,  Comm. Math. Phys. 353 (2017), no. 1, 229--251], where the existence and stability of standing waves for the supercritical NLS with a partial confinement were established.
	\end{abstract}
	
	\maketitle
	
	\section{Introduction}
	\label{S1}
	\setcounter{equation}{0}
	\subsection{Physical motivation}
	Bose-Einstein condensate (BEC) is a state of matter which describes a phenomenon that near absolute zero temperature all atoms lose their individual properties and condense into a macroscopic coherent super atom wave. This phenomenon was predicted by Bose and Einstein in 1925, but the experimental realization was not possible until 1995 by JILA \cite{AEMWC} and MIT \cite{DMADDKK} groups (2001 Nobel Prize in physics attributed to Cornell, Wieman, and Ketterle). Since 1995, the study of Bose-Einstein condensation has become one of the most active areas not only in physics but also in mathematics. An interesting application of Bose-Einstein condensates (BEC) is its application to the superfluidity and superconductivity. The key issue is to study the existence of quantized vortices which are well-known signatures of superfluidity (see e.g., \cite{Aftalion} for a broad introduction on these phenomena). Currently, the most popular way to generate quantized vortices from BEC is to impose a laser beam rotating with an angular velocity on the magnetic trap holding the atoms to create a harmonic anisotropic potential (see e.g., \cite{WH, MCWD}).
	
	In the mean-field approximation, the rotational trapped BEC is well described by the macroscopic wave function $\psi(t,x)$ whose evolution is governed by the Gross-Pitaevskii equation (GPE) with an angular momentum rotational term (see e.g., \cite{CD, FCS, GP, Fetter, BC}),
	namely
	\[
	i\hbar \partial_t \psi(t,x) = \left(-\frac{\hbar^2}{2m} \Delta + V(x) - \Omega L_z + M g|\psi(t,x)|^2 \right) \psi(t,x)
	\]
	where $t$ is the time variable, $x = (x_1, x_2, x_3) \in \R^3$ is the spatial coordinate vector,
	$\hbar$ is the Planck constant, $M$ is the number of atoms in the condensate, $m$ is the atomic mass, and $g=\frac{4\pi \hbar^2 a_s}{m}$ with the $s$-wave scattering length $a_s$ (positive for repulsive interactions and negative for attractive interactions). The external potential
	\[
	V(x) =
	\frac{m}{2} \sum_{j=1}^3 \gamma_j^2 x^2_j
	\]
	is the harmonic potential with trap frequencies $\gamma_j > 0$ for $j = 1,2,3$. The angular momentum
	operator
	\[
	L_z = i\hbar (x_2 \partial_{x_1} - x_1 \partial_{x_2})
	\]
	with a rotational speed $\Omega>0$. 
	\subsection{Mathematical framework}
	The mathematical study of the rotating GPE with repulsive interactions has been extensively studied in many works (see e.g., \cite{LS, Seiringer,IM, HHL-MMAS, HHL-JMP} and references therein). In this paper, we are interested in the rotating GPE with attractive interactions. For the mathematical analysis, it is convenient to consider the Cauchy problem for the following more general nonlinear Schr\"odinger equation (NLS) with rotation in the dimensionless form
	\begin{align} \label{RNLS}
	\left\{
	\renewcommand*{\arraystretch}{1.2}
	\begin{array}{rcl}
	i\partial_t u + \frac{1}{2}\Delta u &=& V u - \Omega L_z u- |u|^{p-1}u, \quad (t,x)\in \R\times \R^N, \\
	\left. u \right|_{t=0} &=& u_0,
	\end{array}
	\right.	
	\end{align}
	where $N\geq 2$, $p>1$, $V$ is a harmonic potential of the form
	\begin{align} \label{defi-poten-V}
	V(x) = \frac{1}{2} \sum_{j=1}^{N} \gamma_j^2 x_j^2, \quad \gamma_j>0, \quad j=1, \cdots, N
	\end{align}
	and 
	\[
	L_z:= i (x_2 \partial_{x_1} - x_1 \partial_{x_2}) 
	\]
	is the rotational operator, and $\Omega >0$ is the rotational speed. 
	
	There are two important physical quantities which are formally conserved by the time-evolution associated to \eqref{RNLS}:
	\begin{align*}
	M(u(t))&:= \|u(t)\|^2_{L^2} = M(u_0), \tag{Mass} \\
	E_\Omega(u(t)) &:= \frac{1}{2} \|\nabla u(t)\|^2_{L^2} + \Vc(u(t)) - L_\Omega(u(t)) - \frac{2}{p+1} \|u(t)\|^{p+1}_{L^{p+1}} = E_\Omega(u_0), \tag{Energy}
	\end{align*}
	where 
	\begin{align} \label{V-f}
	\Vc(f):= \int_{\R^N} V(x)|f(x)|^2 dx
	\end{align}
	is the potential energy and
	\begin{align} \label{L-gamma-f}
	L_\Omega(f):= \Omega \int_{\R^N} \overline{f}(x) L_z f(x) dx
	\end{align}
	is the angular momentum.
	
	Motivated by the fact that physicists are often interested in normalized standing waves for \eqref{RNLS}, we study prescribed mass standing waves for \eqref{RNLS}, i.e., solutions to \eqref{RNLS} of the form $u(t,x)=e^{i\omega t} \phi(x)$ having a given mass, where $\omega \in \R$ is a frequency and $\phi$ is a non-trivial solution to the time independent equation 
	\begin{align} \label{ell-equ-intro}
	-\frac{1}{2} \Delta \phi + V\phi - |\phi|^{p-1} \phi - \Omega L_z \phi + \omega \phi =0.
	\end{align} 
	To show the existence of prescribed mass standing waves, we look for critical points of the energy functional under a mass constraint. More precisely, we consider the minimizing problem: for $c>0$,
	\begin{align} \label{I-Omega-c}
	I_\Omega(c):= \inf \left\{ E_\Omega(f) \ : \ f \in S(c) \right\}, 
	\end{align}
	where $S(c):= \left\{ f \in X \ : \ M(f)=c\right\}$. Here $X$ is the functional space in which the energy is well-defined. 
	
	To our knowledge, there are several works devoted to the existence and stability of prescribed mass standing waves for \eqref{RNLS} with low rotational speed\footnote{Since we consider the rotation on the $(x_1,x_2)$-plane, the rotation speed should be compared only with the trapping frequencies in the $x_1$ and $x_2$ directions.}, i.e., $0<\Omega<\min\{\gamma_1,\gamma_2\}$. In this setting, the energy functional is well-defined for $f \in \Sigma$, where
	\begin{align} \label{Sigma}
	\Sigma:=\left\{ f \in H^1(\R^N) \ : \ |x| f \in L^2(\R^N)\right\},
	\end{align}
	hence $X\equiv \Sigma$. It was proved by Antonelli-Marahrens-Sparber \cite[Lemma 3.1]{AMS} that \eqref{RNLS} is locally well-posed in $\Sigma$. More precisely, for
	\[
	N\geq 2, \quad 1<p<1+\frac{4}{N-2}, \quad \Omega>0, \quad u_0 \in \Sigma,
	\]
	there exist $T_*, T^* \in (0, \infty]$ and a unique maximal solution $u \in C((-T_*,T^*), \Sigma)$ to \eqref{RNLS}. The maximal time of existence satisfies the blow-up alternative: if $T^*<\infty$ (resp. $T_*<\infty$), then
	\begin{align*} 
	\lim_{t\nearrow T^*} \|\nabla u(t)\|_{L^2} =\infty \quad \left(\text{resp. } \lim_{t\searrow -T_*} \|\nabla u(t)\|_{L^2} = \infty\right).
	\end{align*}
	Moreover, there are conservation laws of mass and energy, i.e., $M(u(t))=M(u_0)$ and $E_\Omega(u(t))=E_\Omega(u(t))$ for all $t\in (-T_*,T^*)$. In addition, the angular momentum $L_\Omega(u(t))$ is real-valued and satisfies
	\begin{align} \label{momen}
	L_\Omega(u(t)) + \Omega \int_0^t \int_{\R^N} i |u(s,x)|^2 L_z V(x) dx ds = L_\Omega(u_0)
	\end{align}
	for all $t\in (-T_*,T^*)$. 
	
	In \cite{ANS}, Arbunich-Nenciu-Sparber showed the existence and stability of prescribed mass standing waves for \eqref{RNLS} with mass-subcritical nonlinearity. In particular, they proved that for 
	\[
	0<\Omega<\min\{\gamma_1,\gamma_2\}, \quad 1<p<1+\frac{4}{N}, \quad c>0,
	\]
	there exists a minimizer for $I_\Omega(c)$. Moreover, the set of minimizers for $I_\Omega(c)$ denoted by
	\[
	\Mcal_\Omega(c):= \left\{ \phi \in S(c) \ : \ E_\Omega(\phi) = I_\Omega(c)\right\}
	\]
	is orbitally stable under the flow of \eqref{RNLS} in the sense that for any $\vareps>0$, there exists $\delta>0$ such that for any initial data $u_0 \in \Sigma$ satisfying 
	\[
	\inf_{\phi \in \Mcal_\Omega(c)} \|u_0-\phi\|_{\Sigma} <\delta,
	\]
	the corresponding solution to \eqref{RNLS} exists globally in time and satisfies
	\begin{align} \label{orbi-stab}
	\inf_{\phi \in \Mcal_\Omega(c)} \|u(t)-\phi\|_{\Sigma}<\vareps, \quad \forall t\in \R.
	\end{align}
	
	In the mass-critical case, i.e., $p=1+\frac{4}{N}$ and isotropic harmonic potential, i.e., $\gamma_1=\cdots=\gamma_N=\gamma$, the existence and stability of prescribed mass standing waves for \eqref{RNLS} were established in \cite{LL-JAMI, BHHZ}. More precisely, they proved that for
	\[
	0<\Omega<\gamma, \quad p=1+\frac{4}{N}, \quad 0<c<M(Q),
	\]
	where $Q$ is the unique positive radial solution to 
	\begin{align} \label{Q}
	-\frac{1}{2} \Delta Q + Q- |Q|^{\frac{4}{N}}Q=0,
	\end{align}
	then there exists a minimizer for $I_\Omega(c)$. In addition, the set of minimizers $\Mcal_\Omega(c)$ is orbitally stable under the flow of \eqref{RNLS} in the sense of \eqref{orbi-stab}. We also mention recent works \cite{LNR, GLY, Guo, GLP} for the limiting behavior of minimizers for $I_\Omega(c)$ when $c \nearrow M(Q)$. 
	
	In the mass-supercritical case, i.e., $1+\frac{4}{N}<p<1+\frac{4}{N-2}$, a standard scaling argument shows that the energy functional is no longer bounded from below on $S(c)$. Inspired by an idea of Bellazzini-Boussa\"id-Jeanjean-Visciglia \cite{BBJV}, recent works \cite{LY, AH} study the local minimization problem: for $c,m>0$,
	\begin{align} \label{I-Omega-c-m}
	I^m_\Omega(c):= \inf \left\{ E_\Omega(f) \ : f \in S(c) \cap B_\Omega(m)\right\},
	\end{align}
	where
	\[
	B_\Omega(m):= \left\{ f \in \Sigma \ : \ \|\nabla f\|^2_{L^2} + 2\Vc(f) - 2L_\Omega(f) \leq m \right\}.
	\]
	They proved that for 
	\[
	0<\Omega <\min\{\gamma_1,\gamma_2\}, \quad 1+\frac{4}{N}<p<1+\frac{4}{N-2}, \quad m>0,
	\]
	there exists $c_0=c_0(m)>0$ sufficiently small such that for all $0<c<c_0$, there exists a minimizer for $I_\Omega^m(c)$ and the set of minimizers for $I_\Omega^m(c)$ is orbitally stable under the flow of \eqref{RNLS} in the sense of \eqref{orbi-stab}.
	
	\begin{remark} \label{rem-low-spe}
		The results in \cite{ANS, LL-JAMI, BHHZ, LY, AH} were stated for $0<\Omega<\min_{1\leq j\leq N}\gamma_j$. However, after a careful look (see Lemma \ref{lem-equiv-norm}), we can prove the following equivalent norm 
		\begin{align} \label{equi-norm-intro}
		\|\nabla f\|^2_{L^2} + 2\Vc(f)  - 2L_\Omega(f) \simeq \|\nabla f\|^2_{L^2} + \|xf\|^2_{L^2}
		\end{align}
		for all $0<\Omega<\min\{\gamma_1,\gamma_2\}$. Thanks to this norm equivalence and the compact embedding 
		\begin{align} \label{comp-embe}
		\Hs^1 \hookrightarrow L^r, \quad \forall 2\leq r<\frac{2N}{N-2},
		\end{align}
		the above-mentioned results of \cite{ANS, LL-JAMI, BHHZ, LY, AH} actually holds for all $0<\Omega<\min\{\gamma_1,\gamma_2\}$. 
	\end{remark}

	In the case of high rotational speed, we are only aware of a non-existence of minimizers for $I_\Omega(c)$ due to Bao-Wang-Markowich \cite{BWM} and Cai \cite[Theorem 4.1]{Cai} (see also \cite[Theorem 5.2]{BC}). In particular, they proved that for
	\[
	\Omega > \min \{\gamma_1, \gamma_2\}, \quad 1<p<1+\frac{4}{N-2}, \quad c>0,
	\]
	there is no minimizer for $I_\Omega(c)$, i.e., $I_\Omega(c)=-\infty$. 
	
	\subsection{Main results}
	To the best of our knowledge, there is no result concerning the existence and stability of prescribed mass standing waves for \eqref{RNLS} with the critical rotational speed $\Omega =\min \{\gamma_1, \gamma_2\}$. In the present paper, we focus our attention to the case of axially symmetric harmonic potential, namely
	\begin{align} \label{axia-pote} 
	\gamma_1=\gamma_2=:\gamma
	\end{align} 
	and our main purpose is to study the existence and stability of prescribed mass standing waves for \eqref{RNLS} with the critical rotational speed $\Omega = \gamma$. 
	
	In the presence of critical rotational speed $\Omega =\gamma$, the main difficulty in proving the existence and stability of prescribed mass standing waves for \eqref{RNLS} comes from the fact that there is no equivalent norm between $\|\nabla f\|^2_{L^2} + 2 \Vc(f) - 2L_\gamma(f)$ and $\|\nabla f\|^2_{L^2} + \|xf\|^2_{L^2}$. Thus the compact embedding \eqref{comp-embe} does not help to show the existence of minimizers for $I_\gamma(c)$. To overcome the difficulty, we rewrite \eqref{RNLS} as
	\begin{align} \label{RNLS-mag}
	\left\{
	\renewcommand*{\arraystretch}{1.2}
	\begin{array}{rcl}
	i\partial_t u + \frac{1}{2}(\nabla-iA)^2 u &=& V_\gamma u - |u|^{p-1}u, \quad (t,x)\in \R\times \R^N, \\
	\left. u \right|_{t=0} &=& u_0,
	\end{array}
	\right.	
	\end{align}
	where 
	\begin{align} \label{defi-V-gamma}
	A(x)= \gamma(-x_2, x_1, 0, \cdots, 0), \quad V_\gamma(x) := \frac{1}{2} \sum_{j=3}^N \gamma_j^2 x_j^2.
	\end{align}
	When $N=2$ or $V_\gamma \equiv 0$, it is the magnetic Schr\"odinger equation with no external potential which has been studied in \cite{EL, CE, Ribeiro}. When $N\geq 3$, \eqref{RNLS-mag} can be viewed as NLS with a constant magnetic field and a partial harmonic confinement in $x_3, \cdots, x_N$ directions. Note that when $A \equiv 0$ and $N=3$, this equation has been studied recently in \cite{BBJV}. 
	
	Under this setting, the energy functional now becomes
	\[
	E_\gamma(u(t)) = \frac{1}{2} \|(\nabla -iA) u(t)\|^2_{L^2} + \int_{\R^N} V_\gamma(x) |u(t,x)|^2 dx - \frac{2}{p+1} \|u(t)\|^{p+1}_{L^{p+1}} =E_\gamma(u_0)
	\]
	and it is well-defined on $X=\Sigma_\gamma$, where
	\begin{align} \label{Sigma-gamma}
	\Sigma_\gamma:= \left\{ f \in H^1_A(\R^N) \ : \ \int_{\R^N} V_\gamma(x)|f(x)|^2 dx <\infty \right\}
	\end{align}
	is equipped with the norm
	\begin{align} \label{f-Sigma-gamma}
	\|f\|^2_{\Sigma_\gamma}:= \|(\nabla-iA)f\|^2_{L^2} + \int_{\R^N} V_\gamma(x)|f(x)|^2 dx + \|f\|^2_{L^2}. 
	\end{align}
	Here $H^1_A(\R^N)$ is the magnetic Sobolev space defined by
	\[
	H^1_A:= \left\{ f \in L^2(\R^N) \ : \ |(\nabla-iA)f| \in L^2(\R^N)\right\}.
	\]
	It was proved by Yajima \cite{Yajima} that the linear Schr\"odinger operator $e^{i t \Hc}$ with $\Hc:=\frac{1}{2} (\nabla-iA)^2 - V_\gamma$ can be expressed, for $|t|<\delta$ with some $\delta>0$, in terms of a Fourier integral operator of the form
		\[
		e^{it \Hc} f(x) = (2\pi it)^{-\frac{N}{2}} \int e^{iS(t,x,y)} a(t,x,y) f(y) dy,
		\]
		where $S$ and $a$ are $C^1$ in $(t,x,y)$ and $C^\infty$ in $(x,y)$, and $|\partial^\alpha_x \partial^\beta_y a(t,x,y)| \leq C_{\alpha\beta}$ for all multi-indices $\alpha, \beta$. From this, we obtain the dispersive estimate for $e^{it\Hc}$ for every $|t| <\delta$. Thanks to this dispersive estimate, the standard argument (see \cite[Chapter 9]{Cazenave}) yields the local well-posedness for \eqref{RNLS-mag} with initial data in $\Sigma_\gamma$.
	In particular, we have that for 
	\[
	N\geq 2, \quad 1<p<1+\frac{4}{N-2}, \quad u_0 \in \Sigma_\gamma,
	\]
	there exists a unique maximal solution $u \in C((-T_*,T^*), \Sigma_\gamma)$ to \eqref{RNLS-mag}. The local solution satisfies the conservation of mass and energy, i.e., $M(u(t))=M(u_0)$ and $E_\gamma(u(t))=E_\gamma(u_0)$ for all $t\in (-T_*, T^*)$. In addition, the maximal time of existence satisfies the blow-up alternative: if $T^*<\infty$ (resp. $T_*<\infty$), then
	\begin{align} \label{blow-alte-gamma}
	\lim_{t\nearrow T^*} \|u(t)\|_{\Sigma_\gamma} =\infty \quad \left(\text{resp. } \lim_{t\searrow-T_*} \|u(t)\|_{\Sigma_\gamma}=\infty \right).
	\end{align}
	
	\begin{remark}
		It was noticed in \cite[Remark 3.3]{AMS} that \eqref{RNLS} with $V$ satisfying \eqref{axia-pote} can be transformed back to NLS with no rotation. More precisely, by the change of variable
		\begin{align} \label{chan-vari}
		v(t,x) := u(t,x_1 \cos(\Omega t) + x_2 \sin(\Omega t), -x_1 \sin(\Omega t) + x_2 \cos(\Omega t), x_3, \cdots, x_N),
		\end{align}
		we see that $u$ solves \eqref{RNLS} if and only if $v$ solves
		\begin{align} \label{NLS}
		\left\{
		\renewcommand*{\arraystretch}{1.2}
		\begin{array}{rcl}
		i\partial_t v + \frac{1}{2}\Delta v &=& V v - |v|^{p-1}v, \quad (t,x)\in \R\times \R^N, \\
		\left. v \right|_{t=0} &=& u_0,
		\end{array}
		\right.	
		\end{align}
		where $V$ is as in \eqref{defi-poten-V}. The transformation \eqref{chan-vari} preserves the Lebesgue norms as well as the kinetic energy. Thus dynamics of solutions to \eqref{RNLS} with data in $\Sigma$ can be inferred from that of \eqref{NLS}. However, since $\Sigma$ is only a subspace of $\Sigma_\gamma$ (see Lemma \ref{lem-Sigma}) and \eqref{NLS} is not well-posed in $\Sigma_\gamma$, dynamics of solutions to \eqref{RNLS} with data in $\Sigma_\gamma$ do not simply follow from that of \eqref{NLS}.
	\end{remark}

	Our first result is the following existence and stability of prescribed mass standing waves for \eqref{RNLS-mag} in the mass-subcritical case.
	
	\begin{theorem}[Standing waves in the mass-subcritical regime] \label{theo-mass-sub}
		Let $N\geq 2$, $1<p<1+\frac{4}{N}$, and $V$ be as in \eqref{defi-poten-V} satisfying \eqref{axia-pote}. Assume that $\Omega=\gamma$. Then for any $c>0$, there exists $\phi \in \Sigma_\gamma$ such that $E_\gamma(\phi) = I_\gamma(c)$ and $M(\phi)=c$. In particular, $u(t,x)=e^{i\omega t} \phi(x)$ is a solution to \eqref{RNLS} with $\omega$ the corresponding Lagrange multiplier. Moreover, the set of minimizers for $I_\gamma(c)$ denoted by
		\[
		\Mcal_\gamma(c):= \left\{ \phi \in S(c) \ : \ E_\gamma(\phi) = I_\gamma(c)\right\}
		\]
		is orbitally stable under the flow of \eqref{RNLS} in the sense that for any $\varepsilon>0$, there exists $\delta>0$ such that for any initial data $u_0 \in \Hs^1$ satisfying
		\[
		\inf_{\phi\in \Gc_\gamma(c)} \|u_0 - \phi\|_{\Sigma_\gamma} <\delta,
		\]
		the corresponding solution to \eqref{RNLS} exists globally in time and satisfies
		\[
		\inf_{\phi \in \Gc_\gamma(c)} \inf_{y \in \Theta} \|e^{iA(y)\cdot \cdotb}u(t, \cdotb+y)-\phi\|_{\Sigma_\gamma} <\varepsilon, \quad \forall t \in \R,
		\]
		where $\Theta:= \R^2 \times \{0\}_{\R^{N-2}}$.
	\end{theorem}
	
	When $N=2$, i.e., $V_\gamma=0$, the existence of minimizers for $I_\gamma(c)$ was proved by Esteban-Lions \cite{EL} and the orbital stability of $\Mcal_\gamma(c)$ was showed by Cazenave-Esteban \cite{CE}. The proof of the existence result in \cite{EL} is based on a variant of the celebrated concentration-compactness principle adapted to the magnetic Sobolev space $H^1_A$. See also a recent paper \cite{CDH} for another usage of this magnetic concentration-compactness principle. In \cite{EL, CDH}, the exclusion of the vanishing scenario is based on the negativity of the minimization. When there is no external potential, this negativity can be achieved by using a suitable scaling argument. However, in our setting (especially when $N\geq 3$), this scaling argument does not work due to the presence of a partial harmonic confinement $V_\gamma$ and it may happen that $I_\gamma(c)$ is non-negative. Thus the argument given in \cite{CE} is not applicable to treat our problem. The proof of the existence part in Theorem \ref{theo-mass-sub} is based on an idea of \cite{BBJV} which does not use the concentration-compactness principle. More precisely, by making use of the diamagnetic inequality (see e.g., \cite{LL}), we shall prove in Lemma \ref{lem-weak-conv} a weak convergence result for bounded sequences in $\Sigma_\gamma$ having $L^{p+1}$-norm bounded away from zero. This allows us to rule out the vanishing possibility. To show the boundedness away from zero of the $L^{p+1}$-norm of every minimizing sequence for $I_\gamma(c)$, we proceed in two steps. First, we show (see Lemma \ref{lem-omega-gamma}) that 
	\[
	\omega^0_\gamma=\omega^0=\frac{1}{2} \sum_{j=1}^N \gamma_j,
	\] 
	where
	\begin{align} \label{omega-0}
	\omega^0 := \inf \left\{ \frac{1}{2} \|\nabla f\|^2_{L^2} + \Vc(f) \ : \ f \in \Sigma, M(f)=1 \right\}
	\end{align}
	and 
	\begin{align} \label{omega-gamma-0}
	\omega^0_\gamma:= \inf \left\{\frac{1}{2} \|(\nabla-iA) f\|^2_{L^2}+ \int_{\R^N} V_\gamma(x) |f(x)|^2 dx \ : \ f \in \Sigma_\gamma, M(f) =1 \right\}.
	\end{align}
	This is done by using an argument of \cite{BBJV} and an estimate of the magnetic Sobolev norm due to \cite{EL}. Second, we argue by contradiction that if there exists a minimizing sequence $(f_n)_n$ for $I_\gamma(c)$ satisfying $\lim_{n\rightarrow \infty} \|f_n\|_{L^{p+1}}=0$, then we must have $I_\gamma(c) \geq \omega^0_\gamma c$. This, however, is a contradiction thanks to a suitable choice of test function. Once the vanishing is excluded, an application of the Brezis-Lieb's lemma (see e.g., \cite{BL}) shows the existence of minimizers for $I_\gamma(c)$. The proof of the orbital stability part in Theorem \ref{theo-mass-sub} relies on the contradiction argument due to \cite{CE}. We refer the reader to Section \ref{S2} for more details. 
	
	Our next result concerns the existence and stability of prescribed mass standing waves for \eqref{RNLS-mag} in the mass-critical case.
	
	\begin{theorem} [Standing waves in the mass-critical case] \label{theo-mass-cri}
		Let $N\geq 2$, $p=1+\frac{4}{N}$, and $V$ be as in \eqref{defi-poten-V} satisfying \eqref{axia-pote}. Assume that $\Omega=\gamma$. Let $0<c<M(Q)$, where $Q$ is the unique positive radial solution to \eqref{Q}. Then there exists $\phi \in \Sigma_\gamma$ such that $E_\gamma(\phi) = I_\gamma(c)$ and $M(\phi)=c$. In particular, $u(t,x)=e^{i\omega t} \phi(x)$ is a solution to \eqref{RNLS} with $\omega$ the corresponding Lagrange multiplier. Moreover, the set of minimizers for $I_\gamma(c)$ is orbitally stable under the flow of \eqref{RNLS} in the sense  of Theorem \ref{theo-mass-sub}.
	\end{theorem}

	We also have the following non-existence of minimizers for $I_\gamma(c)$. 
	
	\begin{proposition} \label{prop-non-exis}
		Let $N\geq 2$ and $V$ be as in \eqref{defi-poten-V} satisfying \eqref{axia-pote}. Assume that $\Omega=\gamma$. Then there is no minimizer for $I_\gamma(c)$ provided that one of the following conditions holds:
		\begin{itemize} [leftmargin=6mm]
			\item $p=1+\frac{4}{N}$ and $c\geq \|Q\|^2_{L^2}$, where $Q$ is the unique positive radial solution to \eqref{Q}. 
			\item $1+\frac{4}{N}<p<1+\frac{4}{N-2}$ and $c>0$.
		\end{itemize}
	\end{proposition}
	
	\begin{remark}
		After finishing this manuscript, we learn that similar results as in Theorem \ref{theo-mass-cri} and Proposition \ref{prop-non-exis} were proved recently by Guo-Luo-Peng \cite{GLP-cri} in two dimensions. Their proof is based on the concentration-compactness principle in the same spirit of \cite{EL}. Here we give an alternative simple approach that not only avoids the concentration-compactness principle, but also is applicable for higher dimensions where there is a partial harmonic confinement. 
	\end{remark}

	In the mass-supercritical case, by Proposition \ref{prop-non-exis}, it is not possible to look for global minimizers for the energy functional over $S(c)$. As in \eqref{I-Omega-c-m}, we consider the following minimization problem
	\[
	I^m_\gamma(c):= \inf \left\{E_\gamma(f) \ : \ f \in S_\gamma(c) \cap B_\gamma(m)\right\},
	\]
	where
	\[
	B_\gamma(m):= \left\{ f\in \Sigma_\gamma \ :\ \|(\nabla-iA) f\|^2_{L^2} + 2 \int_{\R^N} V_\gamma(x)|f(x)|^2 dx \leq m \right\}.
	\]
	Our next result is the following existence and stability for prescribed mass standing waves for \eqref{RNLS-mag} in the mass-supercritical case.
	
	\begin{theorem} [Standing waves in the mass-supercritical case]  \label{theo-mass-sup}
		Let $N\geq 2$, $1+\frac{4}{N}<p<1+\frac{4}{N-2}$, and $V$ be as in \eqref{defi-poten-V} satisfying \eqref{axia-pote}. Assume that $\Omega=\gamma$. Then for every $m>0$, there exists $c_0=c_0(m)>0$ sufficiently small such that for all $0<c<c_0$:
		\begin{itemize} [leftmargin=6mm]
			\item[(1)] There exists a minimizer $\phi$ for $I^m_\gamma(c)$. Moreover, the set of minimizers for $I^m_\gamma(c)$ defined by
			\[
			\Mcal^m_\gamma(c):= \left\{ \phi \in S(c) \cap B_\gamma(m) \ : \ E_\gamma(\phi) = I^m_\gamma(c)\right\}
			\]
			satisfies
			\[
			\emptyset \ne \Mcal^m_\gamma(c) \subset B_\gamma(m/2).
			\]
			In particular, $u(t,x) = e^{i\omega t} \phi(x)$ is a solution to \eqref{RNLS} with $\omega$ the corresponding Lagrange multiplier satisfying
			\begin{align} \label{est-omega}
			-\omega^0_\gamma < \omega \leq -\omega^0_\gamma \left(1-B m^{\frac{N(p-1)-4}{4}} c^{\frac{4-(N-2)(p-1)}{4}} \right)
			\end{align}
			for some constant $B>0$ independent of $c$ and $m$. In addition, we have
			\begin{align} \label{est-sup}
			\sup_{\phi \in \Mcal^m_\gamma(c)} \|\phi\|^2_{\Sigma_\gamma} = O\left(c+ m^{\frac{N(p-1)}{4}}c^{\frac{4-(N-2)(p-1)}{4}}\right).
			\end{align}
			\item[(2)] The set $\Mcal^m_\gamma(c)$ is orbitally stable under the flow of \eqref{RNLS} in the sense of Theorem \ref{theo-mass-sub}.
		\end{itemize}
	\end{theorem}
	
	\begin{remark}
		When $A \equiv 0$, the existence and stability of prescribed mass standing waves for \eqref{RNLS-mag} were studied in \cite{BBJV} with $N=3$ and $\frac{7}{3}<p<5$. Theorem \ref{theo-mass-sup} extends the result in \cite{BBJV} to the case of a constant magnetic field and a partial harmonic confinement.		
	\end{remark}
	
	The proof of Theorem \ref{theo-mass-sup} is based on an argument of \cite{BBJV}. We first show that for each $m>0$ fixed, there exists $c_0=c_0(m)>0$ sufficiently small such that for all $0<c<c_0$, $S(c) \cap B_\gamma(m) \ne \emptyset$, so $I^m_\gamma(c)>-\infty$. Using this, we see that any minimizing sequence for $I^m_\gamma(c)$ is bounded uniformly in $\Sigma_\gamma$. By the same argument mentioned above, we shall prove that any minimizing sequence for $I^m_\gamma(c)$ has $L^{p+1}$-norm bounded away from zero. This together with the Brezis-Lieb's lemma yield the existence of a minimizer for $I^m_\gamma(c)$. To see that this minimizer is indeed a solution to \eqref{ell-equ-intro} with $\omega$ the corresponding Lagrange multiplier, it suffices to show that this minimizer does not belong to the boundary of $B_\gamma(m)$. This is done by proving that for $0<c<c_0$,
	\begin{align} \label{est-inf-intro}
	\inf \left\{E_\gamma(f) \ : \ f \in S(c) \cap B_\gamma(m/4)\right\} < \inf \left\{ E_\gamma(f) \ : \ f \in S(c) \cap \left( B_\gamma(m) \backslash B_\gamma(m/2)\right)\right\}.
	\end{align}
	The estimate \eqref{est-inf-intro} is different to the one used in\cite{BBJV}, namely
	\begin{align} \label{est-inf-BBJV}
	\inf \left\{E_\gamma(f) \ : \ f \in S(c) \cap B_\gamma(cm/2)\right\} < \inf \left\{ E_\gamma(f) \ : \ f \in S(c) \cap \left( B_\gamma(m) \backslash B_\gamma(cm)\right)\right\}.
	\end{align} 
	Here the notations have been changed to be consistent with ours. If we use \eqref{est-inf-BBJV}, then for $f \in S(c) \cap B_\gamma(cm/2)$, we have from \eqref{omega-gamma-0} that
	\[
	c= M(f) \leq \frac{1}{2\omega^0_\gamma} \left(\|(\nabla-iA) f\|^2_{L^2} + 2\int_{\R^N} V_\gamma(x)|f(x)|^2 dx \right) \leq \frac{cm}{4\omega^0_\gamma} \quad \text{or} \quad m \geq 4\omega^0_\gamma
	\]
	hence the argument in \cite{BBJV} does not apply to all $m>0$. 
	
	We also remark that the orbital stability given in Theorem \ref{theo-mass-sup} requires the solutions to \eqref{RNLS-mag} exist globally in time for $c>0$ sufficiently small. This result was not showed in \cite{BBJV}. We shall prove this result in Lemma \ref{lem-gwp-mass-sup} and the proof is based on a standard continuity argument. 
	
	Our final result shows that for each $m>0$ fixed and $c>0$ sufficiently small, minimizers for $I^m_\gamma(c)$ are prescribed mass ground states for \eqref{ell-equ-intro}.
	
	\begin{theorem} \label{theo-mass-sup-prop}
		Let $N\geq 2$, $1+\frac{4}{N}<p<1+\frac{4}{N-2}$, $V$ be as in \eqref{defi-poten-V} satisfying \eqref{axia-pote}, and $\Omega =\gamma$. Let $m>0$ be a fixed constant, $c>0$ sufficiently small, and $\phi \in \Mcal^m_\gamma(c)$. Then $\phi$ is a prescribed mass ground state related to \eqref{ell-equ-intro}, i.e.,
		\[
		\left.E_\gamma'\right|_{S(c)}(\phi)=0, \quad E_\gamma(\phi) = \inf \left\{ E_\gamma(f) \ : \ f \in S(c) \left. E'_\gamma \right|_{S(c)} (f) =0 \right\}.
		\]
	\end{theorem}
	
	We end this paragraph by giving a remark on the rotational NLS with a general (non-axially symmetric) harmonic potential. 
	
	\begin{remark}
		The axially symmetric condition \eqref{axia-pote} plays an essential role in our analysis. If we consider the general (non-axially symmetric) harmonic potential, i.e., $\gamma_1 \ne \gamma_2$, then the arguments presented in the sequel do not work due to the lack of the gauge invariance in the first two variables. Thus the existence and stability of prescribed mass standing waves for \eqref{RNLS} with 
		\[
		\Omega=\min\{\gamma_1, \gamma_2\}
		\]
		are still open problems.
	\end{remark}

	\subsection{Outline of the paper}
	This paper is organized as follows. In Section \ref{S2}, we give some preliminary results which are needed in the sequel. In Section \ref{S3}, we show the existence, non-existence, and stability of prescribed mass standing waves for \eqref{RNLS} with mass-subcritical and mass-critical nonlinearities. Finally, we study the existence, stability, and qualitative properties of prescribed mass standing waves for \eqref{RNLS} with the mass-supercritical nonlinearity in Section \ref{S4}.

	\section{Preliminaries}
	\label{S2}
	\setcounter{equation}{0}
	
	In this section, we recall and prove some preliminary results which are useful in our analysis. Let us start with the following equivalent norm in the case of low rotational speed mentioned in Remark \ref{rem-low-spe}.
	
	\begin{lemma} \label{lem-equiv-norm}
		Let $V$ be as in \eqref{defi-poten-V} satisfying \eqref{axia-pote}. If $0<\Omega <\gamma$, then 
		\begin{align} \label{equi-norm}
		\|\nabla f\|^2_{L^2} + 2\Vc(f)  - 2L_\Omega(f) \simeq \|\nabla f\|^2_{L^2} + \|xf\|^2_{L^2}
		\end{align} 
		for any $f\in \Sigma$.
	\end{lemma}
	
	\begin{proof}
		We first observe from H\"older's inequality and Cauchy-Schwarz' inequality that for any $\delta>0$,
		\begin{align}
		|L_\Omega(f)| &\leq \Omega \left(\|x_1 f\|_{L^2} \|\partial_{x_2} f\|_{L^2} + \|x_2 f\|_{L^2} \|\partial_{x_1} f\|_{L^2}\right) \nonumber \\
		&\leq \Omega \left(\|x_1 f\|^2_{L^2} + \|x_2 f\|^2_{L^2}\right)^{\frac{1}{2}} \left(\|\partial_{x_1} f\|^2_{L^2} + \|\partial_{x_2} f\|^2_{L^2}\right)^{\frac{1}{2}} \nonumber \\
		&\leq \delta \|\nabla f\|^2_{L^2} + \frac{\Omega^2}{4\delta} \int_{\R^N} (x_1^2 + x_2^2)|f(x)|^2 dx. \label{est-L-omega}
		\end{align}
		Applying \eqref{est-L-omega} with $\delta =\frac{1}{2}$, we have
		\begin{align}
		(I):=\|\nabla f\|^2_{L^2} + 2\Vc(f)  - 2L_\Omega(f) &\leq 2\|\nabla f\|^2_{L^2} + \int_{\R^N} \left(2V(x) +\Omega^2(x_1^2 + x_2^2) \right) |f(x)|^2 dx \nonumber\\
		&\leq C_1 \left(\|\nabla f\|^2_{L^2} +  \|x f\|^2_{L^2}\right), \nonumber
		\end{align}
		where
		\[
		C_1:= 2+\max \left\{\gamma^2 + \Omega^2, \gamma_3^2, \cdots, \gamma_N^2\right\}.
		\]
		To see the reverse inequality, we use \eqref{est-L-omega} to have for any $\delta>0$,
		\begin{align*}
		(I) &\geq (1-2\delta)\|\nabla f\|^2_{L^2} + \int_{\R^N} \left(2V(x) -\frac{\Omega^2}{2\delta}(x_1^2 + x_2^2) \right) |f(x)|^2 dx \\
		&= (1-2\delta)\|\nabla f\|^2_{L^2} + \int_{\R^N} \Big[ \Big(\gamma^2- \frac{\Omega^2}{2\delta}\Big) x_1^2 + \Big(\gamma^2-\frac{\Omega^2}{2\delta}\Big) x_2^2 + \sum_{j=3}^N \gamma_j^2 x_j^2 \Big] |f(x)|^2 dx. 
		\end{align*}
		We choose $\delta>0$ such that
		\begin{align*}
		\gamma^2-\frac{\Omega^2}{2\delta} = \frac{\gamma^2-\Omega^2}{2} \quad \text{or} \quad \delta =\frac{1}{2} - \frac{\gamma^2-\Omega^2}{2(\gamma^2+\Omega^2)}>0.
		\end{align*}
		It follows that
		\begin{align*}
		(I) &\geq \frac{\gamma^2-\Omega^2}{\gamma^2+\Omega^2} \|\nabla f\|^2_{L^2} + \min\left\{ \frac{\gamma^2-\Omega^2}{2}, \gamma_3^2, \cdots, \gamma_N^2\right\} \|x f\|^2_{L^2}\nonumber \\
		&\geq C_2 \left(\|\nabla f\|^2_{L^2} +  \|x f\|^2_{L^2}\right),
		\end{align*}
		where
		\begin{align*}
		C_2:= \min \left\{\frac{\gamma^2-\Omega^2}{\gamma^2+\Omega^2}, \frac{\gamma^2-\Omega^2}{2}, \gamma_3^2, \cdots, \gamma_N^2 \right\}.
		\end{align*}
		The proof is complete.
	\end{proof}
	
	In the case of critical rotational speed $\Omega =\gamma$, the equivalent norm \eqref{equi-norm} is no longer available. Thus working on $\Sigma$ is not enough to study \eqref{RNLS-mag}. In particular, we have the following result.
	
	\begin{lemma} \label{lem-Sigma}
		Let $V$ be as in \eqref{defi-poten-V} satisfying \eqref{axia-pote} and $\Omega =\gamma$. Then $\Sigma \equiv \Sigma_A$, where 
		\[
		\Sigma_A:= \left\{ f \in H^1_A(\R^N) \ : \ |x| f \in L^2(\R^N)\right\}
		\]
		is equipped with the norm
		\[
		\|f\|^2_{\Sigma_A} := \|(\nabla -iA)f\|^2_{L^2} + \|xf\|^2_{L^2} + \|f\|^2_{L^2}.
		\]
	\end{lemma}
	
	\begin{proof}
		We first observe that
		\begin{align} \label{obse-A}
		\|(\nabla-iA) f\|^2_{L^2} = \|\nabla f\|^2_{L^2} + 2 \rea \int (\nabla - iA) \cdot \overline{iA f} dx - \|Af\|^2_{L^2}.
		\end{align} 
		By H\"older's and Cauchy-Schwarz' inequalities, we have
		\[
		\|(\nabla -iA)f\|^2_{L^2} \leq \|\nabla f\|^2_{L^2} + \frac{1}{2} \|(\nabla - iA)f\|^2_{L^2} + 3 \|A f\|^2_{L^2}
		\]
		which implies
		\[
		\|(\nabla - iA) f\|^2_{L^2} \leq 2 \|\nabla f\|^2_{L^2} + 6 \|A f\|^2_{L^2} \leq 6 \left( \|\nabla f\|^2_{L^2} + \|xf\|^2_{L^2}\right).
		\]
		On the other hand, by \eqref{obse-A}, we have
		\[
		\|\nabla f\|^2_{L^2} \leq 2 \left( \|(\nabla-iA)f\|^2_{L^2} + \|Af\|^2_{L^2}\right) \leq 2 \left( \|(\nabla-iA)f\|^2_{L^2} + \|xf\|^2_{L^2}\right).
		\]
		The proof is complete.
	\end{proof}

	We next recall some basic properties of the magnetic Sobolev space $H^1_A(\R^N)$.
	
	\begin{lemma}[\cite{EL, LL, CDH}] \label{lem-magn-prop}
		Let $N\geq 2$ and $A$ be as in \eqref{defi-V-gamma}. We have the following properties:
		\begin{itemize}[leftmargin=6mm]
			\item $H^1_A(\R^N)$ is a Hilbert space.
			\item $C^\infty_0(\R^N)$ is dense in $H^1_A(\R^N)$.
			\item $H^1_A(\R^N)$ is continuously embedded in $L^r(\R^N)$ for all $2 \leq r \leq \frac{2N}{N-2}$.
			\item $H^1_A(\R^N) \subset H^1_{\loc}(\R^N) \subset L^r_{\loc}(\R^N)$ for all $2\leq r <\frac{2N}{N-2}$.
			\item Diamagnetic inequality:
			\begin{align} \label{diag-ineq}
			\|\nabla |f|(x)| \leq |(\nabla -iA) f(x)|, \quad \text{a.e. } x \in \R^N. 
			\end{align}
			\item Magnetic Gagliardo-Nirenberg inequality: for $2\leq r<\frac{2N}{N-2}$,
			\begin{align} \label{mag-GN-ineq}
			\|f\|^r_{L^r} \leq C_r \|(\nabla-iA) f\|^{\frac{N(r-2)}{2}}_{L^2} \|f\|^{\frac{N+2-(N-2)r}{2}}_{L^2}, \quad \forall f \in H^1_A(\R^N),
			\end{align}
			where the optimal constant $C_r$ is the same as the sharp constant in the standard Gagliardo-Nirenberg inequality
			\[
			\|f\|^r_{L^r} \leq C_r \|(\nabla-iA) f\|^{\frac{N(r-2)}{2}}_{L^2} \|f\|^{\frac{N+2-(N-2)r}{2}}_{L^2}, \quad \forall f \in H^1(\R^N).
			\]
			Moreover, the equality in \eqref{mag-GN-ineq} cannot be achieved.
		\end{itemize}
	\end{lemma}
	
	For the proof of the sharp magnetic Gagliardo-Nirenberg inequality \eqref{mag-GN-ineq} and its non-attainability, we refer the reader to \cite[Lemma 3.2]{CDH}.

	We next recall the following $L^2$-bound of the magnetic Sobolev norm due to \cite[Proposition 2.2]{EL}.
	
	\begin{lemma}[\cite{EL}]
		Let $A=(A_1, \cdots, A_N) \in W^{1,\infty}_{\loc}(\R^N, \R^N)$ and $j,k \in \{1, \cdots, N\}$. Then for any $f \in C^\infty_0(\R^N)$, we have
		\[
		\left| \int (\partial_j A_k - \partial_k A_j) f \overline{f} dx \right| \leq \|(\partial_j - iA_j) f\|^2_{L^2} + \|(\partial_k - iA_k) f\|^2_{L^2}.
		\]
		In particular, if $A$ is as in \eqref{defi-V-gamma}, then for any $f \in C^\infty_0(\R^N)$, we have
		\begin{align} \label{mag-L2-boun}
		2\gamma \|f\|^2_{L^2} \leq \|(\partial_1-iA_1) f\|^2_{L^2} + \|(\partial_2-iA_2) f\|^2_{L^2}.
		\end{align}
		In addition, if $N=2$, then \eqref{mag-L2-boun} is achieved by $f(x)=\sqrt{\frac{\gamma}{\pi}} e^{-\frac{\gamma}{2} |x|^2}$.
	\end{lemma}

	Thanks to \eqref{mag-L2-boun}, we have the following result.
	
	\begin{lemma} \label{lem-omega-gamma}
		Let $V$ be as in \eqref{defi-poten-V} satisfying \eqref{axia-pote}. Let $\omega^0$ and $\omega^0_\gamma$ be as in \eqref{omega-0} and \eqref{omega-gamma-0} respectively. Then we have 
			\begin{align} \label{omega-gamma}
			\omega^0_\gamma=\omega^0=\frac{1}{2} \sum_{j=1}^N \gamma_j.
			\end{align}
	\end{lemma}
			
			\begin{proof}
				We first observe that $\omega^0$ is the simple first eigenvalue of the multi-dimensional harmonic oscillator $-\frac{1}{2}\Delta + V$. It is well-known that $\omega^0=\frac{1}{2} \sum_{j=1}^N \gamma_j$ and the corresponding eigenfunction to $\omega^0$ is
				\begin{align}
				\Phi(x):= \pi^{-\frac{N}{4}} \left( \prod_{j=1}^N \sqrt{\gamma_j}\right)^{\frac{1}{2}} e^{-\frac{1}{2}\sum_{j=1}^N \gamma_j x_j^2}.
				\end{align}
				We have $M(\Phi)=1$ and as $\Phi$ is real-valued,
				\begin{align*}
				\frac{1}{2}\|(\nabla-iA) \Phi\|^2_{L^2(\R^N)} + \int_{\R^N} V_\gamma(x)|\Phi(x)|^2 dx = \frac{1}{2} \|\nabla \Phi\|^2_{L^2(\R^N)} + \int_{\R^N} V(x) |\Phi(x)|^2 dx =\frac{1}{2}\sum_{j=1}^N \gamma_j
				\end{align*}
				which shows $\omega^0_\gamma \leq \frac{1}{2}\sum_{j=1}^N \gamma_j$.
				
				To see the reverse inequality, we follow an argument of \cite[Lemma 2.1]{BBJV}. Denote $x=(x_\perp, x_\intercal)$ with $x_\perp=(x_1,x_2) \in \R^2$ and $x_\intercal =(x_3, \cdots, x_N) \in \R^{N-2}$ and consider
				\[
				\lambda_0:= \inf \left\{ \frac{1}{2}\int_{\R^{N-2}} |\nabla g(x_\intercal)|^2 dx_\intercal + \int_{\R^{N-2}} V_\gamma(x_\intercal) |g(x_\intercal)|^2 dx_\intercal  \ : \ g \in \Sigma(\R^{N-2}), \|g\|^2_{L^2(\R^{N-2})} =1 \right\}.
				\]
				We have $\lambda_0 = \frac{1}{2}\sum_{j=3}^N \gamma_j$. Let $\Phi_k(x_\intercal)$ and $\lambda_k$ for $k\geq 0$ be such that
				\begin{align} \label{defi-Phi-k}
				\left(-\frac{1}{2} \Delta_{\R^{N-2}} + V_\gamma\right) \Phi_k = \lambda_k \Phi_k, \quad \|\Phi_k\|^2_{L^2(\R^{N-2})}=1, \quad \lambda_k \leq \lambda_{k+1}, \quad \forall k\geq 0.
				\end{align}
				Note that $(\Phi_k)_{k\geq 0}$ forms an orthonormal basis of $L^2(\R^{N-2})$. For $f \in \Sigma_\gamma$, we write
				\begin{align} \label{decom-f}
				f(x) = \sum_{k\geq 0} f_k(x_\perp) \Phi_k(x_\intercal).
				\end{align}
				If $\|f\|^2_{L^2(\R^N)} =1$, then
				\begin{align} \label{iden-f}
				1= \|f\|^2_{L^2(\R^N)} &= \sum_{k\geq 0} \Big(\int_{\R^2} |f_k(x_\perp)|^2 dx_\perp\Big) \times \Big(\int_{\R^{N-2}} |\Phi_k(x_\intercal)|^2 dx_\intercal \Big) \nonumber \\
				&= \sum_{k\geq0} \|f_k\|^2_{L^2(\R^2)}.
				\end{align}
				It follows from \eqref{decom-f}, \eqref{iden-f}, and the integration by parts that
				\begin{align*}
				\frac{1}{2} \|(\nabla&-iA) f\|^2_{L^2(\R^N)} + \int_{\R^N} V_\gamma(x)|f(x)|^2 dx \\
				&= \frac{1}{2} \sum_{j=1}^2 \int_{\R^N} |(\partial_j - iA_j)f(x)|^2 dx + \frac{1}{2} \sum_{j=3}^N \int_{\R^N} |\partial_j f(x)|^2 dx + \int_{\R^N} V_\gamma(x) |f(x)|^2 dx \\
				&= \frac{1}{2} \sum_{j=1}^2 \int_{\R^N} |(\partial_j - iA_j)f(x)|^2 dx + \int_{\R^N} \left(-\frac{1}{2}\Delta_{\R^{N-2}} + V_\gamma\right) f(x) \overline{f}(x) dx \\
				&=  \frac{1}{2} \sum_{k\geq 0} \sum_{j=1}^2 \int_{\R^2} |(\partial_j -iA_j) f_k(x_\perp)|^2 dx_\perp + \sum_{k\geq 0} \lambda_k \int_{\R^2} |f_k(x_\perp)|^2 dx_\perp \\
				&\geq \left(\gamma+ \lambda_0\right) \sum_{k\geq 0} \|f_k\|^2_{L^2(\R^2)} \\
				&=\gamma+\lambda_0 = \frac{1}{2}\sum_{j=1}^N \gamma_j.
				\end{align*}
				Here we have used \eqref{mag-L2-boun} with $N=2$ to get the last inequality. This shows that $\omega^0_\gamma \geq \frac{1}{2}\sum_{j=1}^N \gamma_j$. The proof is complete.
			\end{proof}
	
	We also need the following weak convergence result which is based on an idea of  \cite[Lemma 3.4]{BBJV}.
	
	\begin{lemma} \label{lem-weak-conv}
		Let $(f_n)_{n\geq 1}$ be a sequence of $\Sigma_\gamma$-functions satisfying
		\begin{align} \label{boun-fn}
		\sup_{n\geq 1} \|f_n\|_{\Sigma_\gamma} <\infty.
		\end{align}
		Assume that there exists $\vareps_0>0$ such that
		\begin{align} \label{low-boun-fn}
		\inf_{n\geq 1} \|f_n\|_{L^{p+1}} \geq \vareps_0
		\end{align}
		for some $1<p<1+\frac{4}{N-2}$. Then there exist $f \in \Sigma_\gamma \backslash \{0\}$ and $(y_n)_{n\geq 1} \subset \R^N$ with $y_n= (y_n^1, y_n^2, 0, \cdots, 0)$ such that up to a subsequence,
		\[
		e^{iA(y_n) \cdot x} f_n(x+y_n) \rightharpoonup f \text{ weakly in } \Sigma_\gamma.
		\]
	\end{lemma}
	
	\begin{proof}
		By interpolation, it follows from \eqref{low-boun-fn} that
		\[
		\inf_{n\geq 1} \|f_n\|_{L^{2+\frac{4}{N}}} \geq \vareps_1>0.
		\]
		In fact, if $p=1+\frac{4}{N}$, then we are done. If $1<p<1+\frac{4}{N}$, we interpolate between $L^2$ and $L^{2+\frac{4}{N}}$. If $1+\frac{4}{N}<p<1+\frac{4}{N-2}$, we interpolate between $L^{2+\frac{4}{N}}$ and $L^r$ for some $p+1<r<2+\frac{4}{N-2}$ and use the Sobolev embedding $H^1 \subset L^r$. 
		
		On the other hand, by Sobolev embedding, we have
		\[
		\|f\|^{2+\frac{4}{N}}_{L^{2+\frac{4}{N}}(Q_k)} \leq C \|f\|^{\frac{4}{N}}_{L^2(Q_k)} \|f\|^2_{H^1(Q_k)},
		\]
		where
		\[
		Q_k := (k, k+1) \times (k, k+1) \times \R^{N-2}, \quad k \in \Z.
		\]
		Taking the sum over $k \in \Z$, we get
		\[
		\|f\|^{2+\frac{4}{N}}_{L^{2+\frac{4}{N}}} \leq C \left( \sup_{k\in \Z} \|f\|_{L^2(Q_k)} \right)^{\frac{4}{N}} \|f\|^2_{H^1}.
		\]
		Replacing $f$ by $|f|$ and using the diamagnetic inequality \eqref{diag-ineq}, we have
		\begin{align*}
		\|f\|^{2+\frac{4}{N}}_{L^{2+\frac{4}{N}}} &\leq C \left( \sup_{k\in \Z} \|f\|_{L^2(Q_k)} \right)^{\frac{4}{N}} \left( \|\nabla |f| \|^2_{L^2} + \|f\|^2_{L^2} \right) \\
		&\leq C \left( \sup_{k\in \Z} \|f\|_{L^2(Q_k)} \right)^{\frac{4}{N}} \left( \|(\nabla-iA) f \|^2_{L^2} + \|f\|^2_{L^2} \right).
		\end{align*}
		Thanks to the above inequality, we infer from \eqref{boun-fn} that there exists $(k_n)_{n\geq 1} \subset \Z$ such that
		\[
		\inf_{n\geq 1} \|f_n\|_{L^2(Q_{k_n})} \geq C.
		\]
		for some constant $C>0$. Set $y_n=(-k_n, -k_n, 0, \cdots, 0)$ and 
		\[
		g_n(x) := e^{iA(y_n) \cdot x} f_n(x+y_n).
		\]
		We have 
		\[
		\|g_n\|^2_{L^2} = \|f_n\|^2_{L^2}, \quad \int_{\R^N} V_\gamma(x) |g_n(x)|^2 dx = \int_{\R^N} V_\gamma(x) |f_n(x)|^2 dx, \quad \|(\nabla-iA)g_n\|^2_{L^2} = \|(\nabla-iA) f_n\|^2_{L^2}.
		\]
		Here we note that $V_\gamma$ is independent of $x_1$ and $x_2$ variables. Moreover, we have
		\[
		\|g_n\|^2_{L^2(Q_0)} = \int_{(0,1)^2 \times \R^{N-2}} |f_n(x_1-k_n, x_2-k_n, x_3, \cdots, x_N)|^2 dx = \int_{Q_{k_n}} |f_n(x)|^2 dx \geq C, \quad \forall n\geq 1.
		\]
		Thus we obtain $\sup_{n\geq 1} \|g_n\|_{\Sigma_\gamma} <\infty$ and
		\[
		\inf_{n\geq 1} \|g_n\|^2_{L^2(Q_0)} \geq C>0.
		\]
		By the compactness property coming from the boundedness in the $x_1, x_2$ variables and the confining potential $V_\gamma$ in the $x_3,\cdots, x_N$ variables, there exists $f \in \Sigma_\gamma \backslash \{0\}$ such that up to a subsequence,
		\[
		g_n \rightharpoonup f \text{ weakly in } \Sigma_\gamma.
		\]
		The proof is complete.
	\end{proof}
	
	\section{Standing waves in the mass-(sub)critical regimes}
	\label{S3}
	\setcounter{equation}{0}
	
	In this section, we give the proofs of the existence and stability of prescribed mass standing waves given in Theorems \ref{theo-mass-sub} and \ref{theo-mass-cri}. Let us start with the following result which plays an important role in ruling out the vanishing scenario.
	
	\begin{lemma} \label{lem-non-vani}
		Let $N\geq 2$, $1<p<1+\frac{4}{N-2}$, $V$ be as in \eqref{defi-poten-V} satisfying \eqref{axia-pote}, and $\Omega=\gamma$. Let $c>0$ and $(f_n)_{n\geq 1}$ be a minimizing sequence for $I_\gamma(c)$. Then there exists $C>0$ such that 
		\[
		\liminf_{n\rightarrow \infty} \|f_n\|_{L^{p+1}} \geq C >0.
		\]
	\end{lemma}
	
	\begin{proof}
		Assume by contradiction that there exists a subsequence still denoted by $(f_n)_{n\geq 1}$ satisfying $\lim_{n\rightarrow \infty} \|f_n\|_{L^{p+1}}=0$. By \eqref{omega-gamma-0} and \eqref{omega-gamma}, we see that
		\begin{align} \label{lowe-boun-I-gamma-c}
		I_\gamma(c)=\lim_{n\rightarrow \infty} E_\gamma(f_n) &= \lim_{n\rightarrow \infty} \frac{1}{2} \|(\nabla-iA) f_n\|^2_{L^2} + \int_{\R^N} V_\gamma(x) |f_n(x)|^2 dx \nonumber \\
		&\geq \lim_{n\rightarrow \infty} \omega^0_\gamma \|f_n\|^2_{L^2} = \omega^0_\gamma c.
		\end{align}
		Denote $x=(x_\perp,x_\intercal)$ with $x_\perp=(x_1,x_2) \in \R^2$ and $x_\intercal=(x_3, \cdots, x_N) \in \R^{N-2}$ and set $g(x_\perp) = \sqrt{\frac{\gamma}{\pi}}e^{-\frac{\gamma}{2}|x_\perp|^2}$. We readily check that
		\[
		\|g\|^2_{L^2(\R^2)}=1, \quad \|(\partial_1-iA_1) g\|^2_{L^2(\R^2)} + \|(\partial_2-iA_2)g\|^2_{L^2(\R^2)} = 2\gamma.
		\]
		Next let $h(x_\intercal) = \sqrt{c} \Phi_0(\intercal)$, where $\Phi_0 \in \Sigma(\R^{N-2})$ is as in \eqref{defi-Phi-k}. We have
		\[
		\|h\|^2_{L^2(\R^{N-2})}=c, \quad \frac{1}{2} \int_{\R^{N-2}} |\nabla h(x_\intercal)|^2dx_\intercal + \int_{\R^{N-2}} V_\gamma(x_\intercal) |h(x_\intercal)|^2 dx_\intercal = \frac{c}{2} \sum_{j=3}^N \gamma_j.
		\]
		Now we define $f(x) = g(x_\perp)h(x_\intercal)$. It follows that
		\[
		\|f\|^2_{L^2} = \|g\|^2_{L^2(\R^2)} \|h\|^2_{L^2(\R^{N-2})} = c
		\]
		and
		\begin{align}
		E_\gamma(f)&= \frac{1}{2}\left(\|(\partial_1-iA_1) g\|^2_{L^2(\R^2)} + \|(\partial_2-iA_2) g\|^2_{L^2(\R^2)} \right) \|h\|^2_{L^2(\R^{N-2})} +\frac{1}{2} \|g\|^2_{L^2(\R^2)} \sum_{j=3}^N \|\partial_j h\|^2_{L^2(\R^{N-2})} \nonumber \\
		&\mathrel{\phantom{=}} + \left(\int_{\R^{N-2}} V_\gamma(x_\intercal) |h(x_\intercal)|^2 dx_\intercal\right) \|g\|^2_{L^2(\R^2)} - \frac{2}{p+1} \|g\|^{p+1}_{L^{p+1}(\R^2)} \|h\|^{p+1}_{L^{p+1}(\R^{N-2})} \nonumber\\
		&=  \left(\gamma+\frac{1}{2}\sum_{j=3}^N \gamma_j\right) c - \frac{2}{p+1} \|g\|^{p+1}_{L^{p+1}(\R^2)} \|h\|^{p+1}_{L^{p+1}(\R^{N-2})} \nonumber\\
		&< \omega^0_\gamma c, \label{est-I-c}
		\end{align}
		where the last inequality comes from \eqref{omega-gamma}. This contradicts \eqref{lowe-boun-I-gamma-c} and the proof is complete.
	\end{proof}
	
	We next have the following global well-posedness which is needed for the stability result.
	
	\begin{lemma} \label{lem-gwp}
		Let $N\geq 2$, $V$ be as in \eqref{defi-poten-V} satisfying \eqref{axia-pote}, $\Omega=\gamma$, and $u_0 \in \Sigma_\gamma$. Then the corresponding solution to \eqref{RNLS} exists globally in time provided that one of the following conditions holds:
		\begin{itemize}[leftmargin=6mm]
			\item $1<p<1+\frac{4}{N}$.
			\item $p=1+\frac{4}{N}$ and $M(u_0) < M(Q)$, where $Q$ is the unique positive radial solution to \eqref{Q}.
		\end{itemize}
	\end{lemma}

	\begin{proof}
		Let $u:(-T_*,T^*)\times \R^N \rightarrow \C$ be the maximal solution to \eqref{RNLS}, hence to \eqref{RNLS-mag}. By the blow-up alternative \eqref{blow-alte-gamma}, it suffices to prove that 
		\begin{align}\label{blow-alte-gamma-appl}
		\sup_{t\in (-T_*,T^*)} \|u(t)\|_{\Sigma_\gamma} <\infty.
		\end{align}
		
		(1) Mass-subcritical case. By the magnetic Gagliardo-Nirenberg inequality \eqref{mag-GN-ineq}, Young's inequality with $0<N(p-1)<4$, and the conservation of mass, we have for any $\vareps>0$,
		\begin{align} \label{est-mass-sub}
		\frac{2}{p+1}\|u(t)\|^{p+1}_{L^{p+1}} &\leq C\|(\nabla-iA) u(t)\|^{\frac{N(p-1)}{2}}_{L^2} \|u(t)\|^{\frac{4-(N-2)(p-1)}{2}}_{L^2} \\
		&\leq \vareps \|(\nabla-iA)u(t)\|^2_{L^2} + C(N,p, \vareps, M(u_0)). \nonumber
		\end{align}
		Thus we get
		\[
		E_\gamma(u(t)) \geq \left(\frac{1}{2}-\vareps\right) \|(\nabla-iA) u(t)\|^2_{L^2} + \int_{\R^N} V_\gamma(x) |u(t,x)|^2 dx - C(N,p,\vareps, M(u_0)).
		\]
		Taking $\vareps=\frac{1}{2}$ and using the conservation of energy, we obtain
		\[
		\frac{1}{2} \|(\nabla-iA) u(t)\|^2_{L^2} + \int_{\R^N} V_\gamma(x)|u(t,x)|^2 dx \leq C(N, p, M(u_0), E_\gamma(u_0))
		\]
		for all $t\in (-T_*,T^*)$. This proves \eqref{blow-alte-gamma-appl}. 
		
		(2) Mass-critical case. By the magnetic Gagliardo-Nirenberg inequality \eqref{mag-GN-ineq} with the optimal constant $C_{2+\frac{4}{N}}=\frac{N+2}{2N}[M(Q)]^{-\frac{2}{N}}$ and the conservation of mass, we have 
		\begin{align} \label{est-mass-cri}
		E_\gamma(u(t)) \geq \frac{1}{2} \left(1-\left( \frac{M(u_0)}{M(Q)} \right)^{\frac{2}{N}}\right) \|(\nabla-iA) u(t)\|^2_{L^2} + \int_{\R^N} V_\gamma(x) |f(x)|^2 dx.
		\end{align}
		As $M(u_0)<M(Q)$, we infer \eqref{blow-alte-gamma-appl}. The proof is complete.
	\end{proof}

	We are now able to prove the existence and stability of prescribed mass standing waves for \eqref{RNLS} given in Theorem \ref{theo-mass-sub}.
	 
	 \begin{proof}[Proof of Theorem \ref{theo-mass-sub}]
	 	The proof is divided into several steps.
	 	
	 	{\bf Step 1.} We first show that $I_\gamma(c)>-\infty$ for all $c>0$. Let $c>0$ and $f \in \Sigma_\gamma$ satisfy $M(f) =c$. Arguing as in \eqref{est-mass-sub}, we have for any $\vareps>0$,
	 	\begin{align*}
	 	\frac{2}{p+1}\|f\|^{p+1}_{L^{p+1}} &\leq C \|(\nabla-iA) f\|^{\frac{N(p-1)}{2}}_{L^2} \|f\|^{\frac{4-(N-2)(p-1)}{2}}_{L^2} \\
	 	&\leq \varepsilon \|(\nabla-iA) f\|^2_{L^2} + C(N,p, \varepsilon, c). 
	 	\end{align*}
	 	It follows that
	 	\begin{align*}
	 	E_\gamma(f) \geq \left(\frac{1}{2}-\varepsilon\right) \|(\nabla-iA) f\|^2_{L^2} +\int_{\R^N} V_\gamma(x)|f(x)|^2 dx - C(N,p, \varepsilon, c).
	 	\end{align*}
	 	By taking $\varepsilon=\frac{1}{4}$, we have
	 	\begin{align} \label{est-E-gamma}
	 	E_\gamma(f) \geq \frac{1}{4} \|(\nabla-iA) f\|^2_{L^2} + \int_{\R^N} V_\gamma(x)|f(x)|^2 dx - C(N,p,c).
	 	\end{align}
	 	Since $V_\gamma\geq 0$, we have $E_\gamma(f) \geq - C(N,p,c)$ for all $f \in \Sigma_\gamma$ satisfying $M(f) =c$. This shows that $I_\Omega(c)$ is well-defined.
	 	
	 	{\bf Step 2.} We will show that there exists a minimizer for $I_\gamma(c)$. To see this, we take $(f_n)_{n\geq 1}$ a minimizing sequence for $I_\gamma(c)$. By \eqref{est-E-gamma}, we have
	 	\[
	 	\frac{1}{4} \|(\nabla-iA) f_n\|^2_{L^2} +\int_{\R^N} V_\gamma(x)|f_n(x)|^2 dx \leq E_\gamma(f_n) + C(N,p,c) \rightarrow I_\gamma(c) + C(N,p,c)
	 	\]
	 	as $n\rightarrow \infty$. We infer that $\sup_{n\geq 1} \|f_n\|_{\Sigma_\gamma} <\infty$. Thanks to Lemma \ref{lem-non-vani}, there exists a subsequence still denoted by $(f_n)_{n\geq 1}$ such that 
	 	\[
	 	\inf_{n\geq 1} \|f_n\|_{L^{p+1}} \geq C>0.
	 	\] 
	 	By Lemma \ref{lem-weak-conv}, there exist $\phi \in \Sigma_\gamma \backslash \{0\}$ and a sequence $(y_n)_{n\geq 1} \subset \R^N$ with $y_n=(y^1_n,y^2_n, 0, \cdots, 0)$ such that up to a subsequence,
	 	\[
	 	g_n(x):=e^{iA(y_n)\cdot x} f_n(x+y_n) \rightharpoonup \phi \text{ weakly in } \Sigma_\gamma.
	 	\]
	 	By the weak convergence in $\Sigma_\gamma$, we have
	 	\[
	 	0<\|\phi\|^2_{L^2} \leq \liminf_{n\rightarrow \infty} \|g_n\|^2_{L^2}= \liminf_{n\rightarrow \infty} \|f_n\|^2_{L^2}=c
	 	\]
	 	and
	 	\begin{align*}
	 	\frac{1}{2} \|(\nabla-iA) \phi\|^2_{L^2} + \int_{\R^N} V_\gamma(x)|\phi(x)|^2 dx &\leq \liminf_{n\rightarrow \infty} \frac{1}{2} \|(\nabla-iA) g_n\|^2_{L^2} + \int_{\R^N} V_\gamma(x)|g_n(x)|^2 dx \\
	 	&= \liminf_{n\rightarrow \infty} \frac{1}{2} \|(\nabla-iA) f_n\|^2_{L^2} + \int_{\R^N} V_\gamma(x)|f_n(x)|^2 dx.
	 	\end{align*}
	 	We claim that $\|\phi\|^2_{L^2}=c$. Assume it is true for the moment. Let us show that $\phi$ is a minimizer for $I_\gamma(c)$. In fact, by the weak convergence in $\Sigma_\gamma$ and $\|\phi\|^2_{L^2} = c= \lim_{n\rightarrow \infty} \|g_n\|^2_{L^2}$, we infer that $g_n \rightarrow \phi$ strongly in $L^2$. Thanks to the magnetic Gagliardo-Nirenberg inequality \eqref{mag-GN-ineq},
	 	we see that $g_n\rightarrow \phi$ strongly in $L^{p+1}$. It follows that
	 	\[
	 	I_\gamma(c) \leq E_\gamma(\phi) \leq \liminf_{n\rightarrow \infty} E_\gamma(g_n) =\liminf_{n\rightarrow \infty} E_\gamma(f_n) = I_\gamma(c).
	 	\]
	 	Therefore $E_\gamma(\phi) = I_\gamma(c)$ or $\phi$ is a minimizer of $I_\gamma(c)$. Moreover, we have $g_n\rightarrow \phi$ strongly in $\Sigma_\gamma$. 
	 	
	 	It remains to prove the claim. Suppose that it is not true, i.e., $0<\|\phi\|^2_{L^2} <c$. By the weak convergence and the Brezis-Lieb's lemma \cite{BL}, we have
	 	\begin{align} \label{BL-est}
	 	\|g_n\|^{p+1}_{L^{p+1}} = \|\phi\|^{p+1}_{L^{p+1}} + \|g_n-\phi\|^{p+1}_{L^{p+1}} + o_n(1), 
	 	\end{align}
	 	where $C_n=o_n(1)$ means $C_n\rightarrow 0$ as $n\rightarrow \infty$. In addition, the weak convergence in $\Sigma_\gamma$ implies that
	 	\begin{align}
	 	\|(\nabla-iA) g_n\|^2_{L^2} &= \|(\nabla-iA) \phi\|^2_{L^2} + \|(\nabla-iA) (g_n-\phi)\|^2_{L^2} + o_n(1), \label{est-1}\\
	 	\int_{\R^N} V_\gamma(x) |g_n(x)|^2 dx &= \int_{\R^N} V_\gamma(x) |\phi(x)|^2 dx + \int_{\R^N} V_\gamma(x) |g_n(x)-\phi(x)|^2 dx + o_n(1). \label{est-2}
	 	\end{align}
	 	In fact, let $h_n:= g_n -\phi$. We see that $h_n \rightharpoonup 0$ weakly in $\Sigma_\gamma$. We compute
	 	\begin{align*}
	 	\|(\nabla-iA)g_n\|^2_{L^2} &= \|(\nabla-iA) (\phi+h_n)\|^2_{L^2} \\
	 	&= \|(\nabla-iA) \phi\|^2_{L^2} + \|(\nabla-iA) h_n\|^2_{L^2} + 2 \rea \int_{\R^N} (\nabla+iA) \overline{\phi} (\nabla-iA) h_n dx.
	 	\end{align*}
	 	Let $\epsilon>0$. Since $C^\infty_0(\R^N)$ is dense in $\Sigma_\gamma(\R^N)$, we take $\varphi \in C^\infty_0(\R^N)$ so that $\|(\nabla+iA)\overline{\phi}-(\nabla-iA)\varphi)\|_{L^2} <\frac{\epsilon}{2M}$, where $M:= \sup_{n\geq 1} \|(\nabla-iA) h_n\|_{L^2}$. Since $h_n \rightharpoonup 0$ weakly in $\Sigma_\gamma(\R^N)$, we see that
	 	\[
	 	\left|\int_{\R^N} (\nabla-iA) \varphi (\nabla-iA) h_n dx\right| \rightarrow 0 \text{ as } n\rightarrow \infty.
	 	\]
	 	Thus there exists $n_0 \in \N$ such that for $n\geq n_0$,
	 	\begin{align*}
	 	\Big|\int_{\R^N} (\nabla+iA) \overline{\phi} &(\nabla-iA) h_n dx\Big| \\
	 	&\leq \left|\int_{\R^N} ((\nabla+iA) \overline{\phi}-(\nabla-iA)\varphi) (\nabla-iA) h_n dx\right| + 	\left|\int_{\R^N} (\nabla-iA) \varphi (\nabla-iA) h_n dx\right| \\
	 	&\leq \|(\nabla+iA)\overline{\phi}-(\nabla-iA)\varphi\|_{L^2} \|(\nabla-iA) h_n\|_{L^2} + \epsilon/2 \\
	 	&<\epsilon.
	 	\end{align*}
	 	This shows \eqref{est-1}. The one for \eqref{est-2} is treated similarly. 
	 	
	 	On the other hand, we have for $\lambda>0$, 
	 	\[
	 	E_\gamma(\lambda \phi) = \lambda^2 E_\gamma(\phi) + \frac{2\lambda^2(1-\lambda^{p-1})}{p+1}\|\phi\|^{p+1}_{L^{p+1}}
	 	\]
	 	or
	 	\begin{align} \label{iden-E-gamma}
	 	E_\gamma(\phi) = \frac{1}{\lambda^2} E_\gamma(\lambda \phi) + \frac{2(\lambda^{p-1}-1)}{p+1} \|\phi\|^{p+1}_{L^{p+1}}.
	 	\end{align}
	 	Applying the above identity to $\lambda_0=\frac{\sqrt{c}}{\|\phi\|_{L^2}} >1$, we have
	 	\[
	 	E_\gamma(\phi) = \frac{\|\phi\|^2_{L^2}}{c} E_\gamma(\lambda_0 \phi) + \frac{2(\lambda_0^{p-1}-1)}{p+1} \|\phi\|^{p+1}_{L^{p+1}} > \frac{\|\phi\|^2_{L^2}}{c} I_\gamma(c)
	 	\] 
	 	as $\|\lambda_0 \phi\|^2_{L^2}=c$ and $\phi \ne 0$. 
	 	
	 	Set $\lambda_n:= \frac{\sqrt{c}}{\|g_n-\phi\|_{L^2}}$. By \eqref{est-2}, we see that $\|g_n-\phi\|^2_{L^2} \rightarrow c-\|\phi\|^2_{L^2}$ as $n\rightarrow \infty$, hence $\lambda_n \rightarrow \frac{\sqrt{c}}{\sqrt{c-\|\phi\|^2_{L^2}}} >1$ as $n\rightarrow \infty$. We infer from \eqref{iden-E-gamma} that
	 	\begin{align*}
	 	\lim_{n\rightarrow \infty} E_\gamma(g_n-\phi) &= \lim_{n\rightarrow \infty} \frac{1}{\lambda_n^2} E_\gamma(\lambda_n (g_n-\phi)) + \frac{2(\lambda_n^{p-1}-1)}{p+1} \|g_n-\phi\|^{p+1}_{L^{p+1}} \\
	 	&\geq \frac{c-\|\phi\|^2_{L^2}}{c} I_\gamma(c).
	 	\end{align*}
	 	On the other hand, by \eqref{BL-est}, \eqref{est-1}, and \eqref{est-2}, we have
	 	\begin{align*}
	 	I_\gamma(c) = \lim_{n\rightarrow \infty} E_\gamma(f_n) = \lim_{n\rightarrow \infty} E_\gamma(g_n) &=  E_\gamma(\phi) + \lim_{n\rightarrow \infty} E_\gamma(g_n-\phi) \\
	 	&>\frac{\|\phi\|^2_{L^2}}{c} I_\gamma(c) + \frac{c-\|\phi\|^2_{L^2}}{c} I_\gamma(c) = I_\gamma(c)
	 	\end{align*}
	 	which is a contradiction. Thus the claim is now proved.
	 	
	 	{\bf Step 3.} Let us now show that the set of minimizers $\Mcal_\gamma(c)$ is orbitally stable under the flow of \eqref{RNLS}. We follow an argument of \cite{CE}. Assume by contradiction that it is not true. Then there exist $\vareps_0$, $\phi_0 \in \Mcal_\gamma(c)$, and a sequence of initial data $(u_{0,n})_{n \geq 1} \subset \Sigma_\gamma$ such that
	 	\begin{align} \label{sta-prof-1}
	 	\lim_{n\rightarrow \infty} \|u_{0,n} -\phi_0\|_{\Sigma_\gamma} =0
	 	\end{align} 
	 	and a sequence of time $(t_n)_{n\geq 1} \subset \R$ such that 
	 	\begin{align} \label{sta-prof-2}
	 	\inf_{\phi \in \Mcal_\gamma(c)} \inf_{y\in \Theta} \|e^{iA(y)\cdot \cdotb}u_n(t_n, \cdotb +y)- \phi\|_{\Sigma_\gamma} \geq \vareps_0,
	 	\end{align}
	 	where $u_n$ is the solution to \eqref{RNLS} with initial data $\left. u_n\right|_{t=0} = u_{0,n}$ and $\Theta:= \R^2 \times \{0\}_{\R^{N-2}}$. Note that the solutions exist globally in time by Lemma \ref{lem-gwp}.
	 	
	 	Since $\phi_0 \in \Mcal_\gamma(c)$, we have $E_\gamma(\phi_0) = I_\gamma(c)$. From \eqref{sta-prof-1} and the Sobolev embedding, we infer that
	 	\[
	 	\|u_{0,n}\|^2_{L^2} \rightarrow \|\phi_0\|_{L^2}^2=c, \quad E_\gamma(u_{0,n}) \rightarrow E_\gamma(\phi_0)=I_\gamma(c) \text{ as } n\rightarrow \infty.
	 	\]
	 	By the conservation laws of mass and energy, we have
	 	\[
	 	\|u_n(t_n)\|^2_{L^2} \rightarrow c, \quad E_\gamma(u_n(t_n)) \rightarrow I_\gamma(c) \text{ as } n \rightarrow \infty.
	 	\]
	 	In particular, $(u_n(t_n))_{n\geq 1}$ is a minimizing sequence for $I_\gamma(c)$. Arguing as in Step 1, we see that up to a subsequence, there exist $\phi \in \Mcal_\gamma(c)$ and $(y_n)_{n\geq 1}\subset \Theta$ such that 
	 	\[
	 	\|e^{iA(y_n)\cdot \cdotb}u_n(t_n, \cdotb+y_n) - \phi\|_{\Sigma_\gamma} \rightarrow 0 \text{ as } n\rightarrow \infty.
	 	\]
	 	This however contradicts \eqref{sta-prof-2}. The proof is complete.	
	 \end{proof}
	 
	We now prove the existence and stability of prescribed mass standing waves for \eqref{RNLS} in the mass-critical case.
	
	\begin{proof}[Proof of Theorem \ref{theo-mass-cri}]
		The proof is similar to that of Theorem \ref{theo-mass-sub}. Thus we only point out the differences. Let $0<c<M(Q)$ and $f \in \Sigma_\gamma$ satisfy $M(f)= c$. From \eqref{est-mass-cri}, we have
		\begin{align} \label{est-E-gamma-mass-cri}
		E_\gamma(f) \geq \frac{1}{2} \left(1-\left(\frac{c}{M(Q)}\right)^{\frac{2}{N}}\right) \|(\nabla-iA) f\|^2_{L^2} + \int_{\R^N} V_\gamma(x)|f(x)|^2 dx
		\end{align}
		Since $0<c<M(Q)$ and $V_\Omega\geq 0$, we have $E_\gamma(f) \geq 0$, hence $I_\gamma(c)$ is well-defined. Let $(f_n)_{n\geq 1}$ be a minimizing sequence for $I_\gamma(c)$. By \eqref{est-E-gamma-mass-cri}, we see that $(f_n)_{n\geq 1}$ is a bounded sequence in $\Sigma_\gamma$. Thanks to Lemma \ref{lem-non-vani}, the existence of minimizers for $I_\gamma(c)$ and the orbital stability of $\Mcal_\gamma(c)$ follow from the same argument as in the proof of Theorem \ref{theo-mass-sub}. We thus omit the details.
	\end{proof}
	
	We end this section by giving the proof of the non-existence of minimizers for $I_\gamma(c)$ given in Proposition \ref{prop-non-exis}.
	
	\begin{proof}[Proof of Proposition \ref{prop-non-exis}]
		(1) Mass-critical case. Let $\varphi \in C^\infty_0(\R^N)$ be radially symmetric satisfying $\varphi(x)=1$ for $|x| \leq 1$. We define
		\[
		f_\lambda(x):= \lambda^{\frac{N}{2}} A_\lambda  \varphi(x) Q_0(\lambda x), \quad \lambda>0,
		\]
		where $Q_0(x)= \frac{Q(x)}{\|Q\|_{L^2}}$ and $A_\lambda>0$ is such that $\|f_\lambda\|^2_{L^2} =c$ for all $\lambda>0$. By definition, we have
		\begin{align*}
		A_\lambda^{-2} =\frac{1}{c} \int_{\R^N} \varphi^2(\lambda^{-1} x) Q_0(x) dx.
		\end{align*}
		Since $Q_0$ decays exponentially at infinity, we see that for $\lambda>0$ sufficiently large and any $\delta>0$,
		\[
		\left| \int_{\R^N} \left(1-\varphi^2(\lambda^{-1} x)\right) Q_0^2(x) dx \right| \lesssim \int_{|x| \geq \lambda} e^{-C |x|} dx \lesssim \int_{|x| \geq \lambda} |x|^{-N-\delta} dx \lesssim \lambda^{-\delta}.
		\]
		In particular, we have $A^2_\lambda = c + O(\lambda^{-\infty})$ as $\lambda \rightarrow \infty$, where $B_\lambda = O(\lambda^{-\infty})$ means that $|B_\lambda| \leq C \lambda^{-\delta}$ for any $\delta>0$ with some constant $C>0$ independent of $\lambda$. Next we have
		\begin{align*}
		\|\nabla f_\lambda\|^2_{L^2} = A_\lambda^2 \Big( \int_{\R^N} |\nabla \varphi(\lambda^{-1} x)|^2 Q_0^2(x) dx &+ \lambda^2 \int_{\R^N} \varphi^2(\lambda^{-1} x) |\nabla Q_0(x)|^2dx \\
		& + 2 \lambda \rea \int_{\R^N} \varphi(\lambda^{-1} x) Q_0(x) \nabla \varphi(\lambda^{-1} x) \cdot \nabla Q_0(x) dx \Big).
		\end{align*}
		As $|\nabla Q_0|$ also decays exponentially at infinity and $A_\lambda^2 = c + O(\lambda^{-\infty})$ as $\lambda \rightarrow \infty$, we infer that
		\[
		\|\nabla f_\lambda\|^2_{L^2} = c \lambda^2  \|\nabla Q_0\|^2_{L^2} + O(\lambda^{-\infty})
		\]
		as $\lambda \rightarrow \infty$. We also have
		\[
		\|f_\lambda\|^{2+\frac{4}{N}}_{L^{2+\frac{4}{N}}} =  c^{1+\frac{2}{N}} \lambda^2 \|Q_0\|^{2+\frac{4}{N}}_{L^{2+\frac{4}{N}}} + O(\lambda^{-\infty})
		\]
		as $\lambda \rightarrow \infty$. Since $f_\lambda$ is radially symmetric, we have $L_\gamma(f_\lambda)=0$. On the other hand, since $\lambda^N Q_0^2(\lambda x)$ converges weakly to the Dirac delta function at zero when $\lambda \rightarrow \infty$, we infer that
		\[
		\Vc(f_\lambda) = A_\lambda^2 \int_{\R^N} V(x) \varphi^2(x) \lambda^N Q_0^2(\lambda x) dx \rightarrow 0
		\]
		as $\lambda \rightarrow \infty$, where $\Vc(f)$ is as in \eqref{V-f}. It follows that
		\begin{align}
		I_\gamma(c) \leq E_\gamma(f_\lambda) &= \frac{1}{2} \|\nabla f_\lambda\|^2_{L^2} + \Vc(f_\lambda) - \frac{N}{N+2} \|f_\lambda\|^{2+\frac{4}{N}}_{L^{2+\frac{4}{N}}} - L_\gamma(f_\lambda) \nonumber \\
		&= c\lambda^2 \left( \frac{1}{2} \|\nabla Q_0\|^2_{L^2} - \frac{N}{N+2} c^{\frac{2}{N}} \|Q_0\|^{2+\frac{4}{N}}_{L^{2+\frac{4}{N}}} \right) + o_\lambda(1) \nonumber \\
		&= \frac{c}{2}\lambda^2 \|\nabla Q_0\|^2_{L^2} \left(1-\left(\frac{c}{M(Q)}\right)^{\frac{2}{N}} \right) + o_\lambda(1) \label{est-I-gamma-c}
		\end{align}
		as $\lambda \rightarrow \infty$, where $B_\lambda=o_\lambda(1)$ means that $|B_\lambda| \rightarrow 0$ as $\lambda \rightarrow \infty$. Here we have used the fact that
		\[
		\frac{N}{N+2}\|Q_0\|^{2+\frac{4}{N}}_{L^{2+\frac{4}{N}}} = \frac{1}{2 \|Q\|^{\frac{4}{N}}_{L^2} } \|\nabla Q_0\|^2_{L^2}
		\]
		which comes from the following Pohozaev's identity (see e.g., \cite{Cazenave}):
		\[
		\|\nabla Q\|^2_{L^2} = \frac{2N}{N+2} \|Q\|^{2+\frac{4}{N}}_{L^{2+\frac{4}{N}}} = N \|Q\|^2_{L^2}.
		\]		
		In the case $c>M(Q)$, letting $\lambda \rightarrow \infty$ in \eqref{est-I-gamma-c}, we get $I_\gamma(c) =-\infty$, hence there is no minimizer for $I_\gamma(c)$. 
		
		In the case $c=M(Q)$, it follows from \eqref{est-I-gamma-c} that $I_\gamma(M(Q)) \leq 0$. On the other hand, by the magnetic Gagliardo-Nirenberg inequality \eqref{mag-GN-ineq}, we have for any $f \in \Sigma_\gamma$ satisfying $M(f) = c=M(Q)$,
		\begin{align*}
		E_\gamma(f) &\geq \frac{1}{2} \|(\nabla-iA) f\|^2_{L^2} + \int_{\R^N} V_\gamma(x) |f(x)|^2 dx - \frac{1}{2} \left(\frac{c}{M(Q)}\right)^{\frac{2}{N}} \|(\nabla-iA) f\|^2_{L^2} \\
		&= \int_{\R^N} V_\gamma(x) |f(x)|^2 dx \geq 0.
		\end{align*}
		This shows that $I_\gamma(M(Q)) \geq 0$, hence $I_\gamma(M(Q)) =0$. We will show that there is no minimizer for $I_\gamma(M(Q))$. Assume by contradiction that there exists a minimizer for $I_\gamma(M(Q))$, says $\phi$. We have
		\begin{align*}
		0= I_\gamma(M(Q)) &= E_\gamma(\phi) \\
		&= \frac{1}{2} \|(\nabla-iA) \phi\|^2_{L^2} + \int_{\R^N} V_\gamma(x) |\phi(x)|^2 dx - \frac{N}{N+2} \|\phi\|^{2+\frac{4}{N}}_{L^{2+\frac{4}{N}}} \\
		&\geq \int_{\R^N} V_\gamma(x) |\phi(x)|^2 dx \geq 0.
		\end{align*}
		It yields that
		\begin{align*}
		\|(\nabla-iA) \phi\|^2_{L^2} = \frac{2N}{N+2} \|\phi\|^{2+\frac{4}{N}}_{L^{2+\frac{4}{N}}}
		\end{align*}
		or $\phi$ is an optimizer of the magnetic Gagliardo-Nirenberg inequality. This however is a contradiction due to Lemma \ref{lem-magn-prop}.
		
		(2) Mass-supercritical case. Let $f \in C^\infty_0(\R^N)$ be radially symmetric with $M(f)=c$. Denote $f_\lambda(x):= \lambda^{\frac{N}{2}} f(\lambda x)$ with $\lambda>0$. We see that $M(f)=M(f_\lambda)=c$ for all $c>0$. Moreover, we have
		\begin{align*}
		E_\gamma(f_\lambda)&= \frac{1}{2} \|\nabla f_\lambda\|^2_{L^2} + \Vc(f_\lambda) - \frac{2}{p+1}\|f_\lambda\|^{p+1}_{L^{p+1}} - L_\gamma(f_\lambda) \\
		&= \frac{\lambda^2}{2} \|\nabla f\|^2_{L^2} + \lambda^{-2} \Vc(f) - \frac{2\lambda^{\frac{N(p-1)}{2}}}{p+1} \|f\|^{p+1}_{L^{p+1}}.
		\end{align*}
		Here $L_\gamma(f_\lambda)=0$ as $f_\lambda$ is radially symmetric. As $N(p-1)>4$, we have $E_\gamma(f_\lambda) \rightarrow -\infty$ as $\lambda \rightarrow \infty$, hence $I_\gamma(c)=-\infty$. The proof is complete.
	\end{proof}

	\section{Standing waves in the mass-supercritical case}
	\label{S4}
	\setcounter{equation}{0}
	
	In this section, we study the existence, stability, and qualitative properties of prescribed mass standing waves for \eqref{RNLS} in the mass-supercritical case. Throughout this section, we denote
	\begin{align} \label{H-gamma-f}
	 H_\gamma(f):= \|(\nabla-iA)f\|^2_{L^2} + 2 \int_{\R^N} V_\gamma(x)|f(x)|^2 dx.
	\end{align}
	Let us start with the following lemmas. 
	
	\begin{lemma} \label{lem-non-vani-sup}
		Let $N\geq 2$, $1+\frac{4}{N}<p<1+\frac{4}{N-2}$, $V$ be as in \eqref{defi-poten-V} satisfying \eqref{axia-pote}, and $\Omega=\gamma$. Then for $m>0$ fixed, there exists $c_0=c_0(m)>0$ sufficiently small such that for all $0<c<c_0$ and any minimizing sequence $(f_n)_{n\geq 1}$ to $I^m_\gamma(c)$, we have
		\[
		\liminf_{n\rightarrow \infty} \|f_n\|_{L^{p+1}} \geq C>0.
		\]
		Moreover, we have $I^m_\gamma(c) < \omega^0_\gamma c$. 
	\end{lemma}
	
	\begin{proof}
		The proof is similar to that of Lemma \ref{lem-non-vani}. Here we note that for $f(x)=g(x_\perp) h(x_\intercal)$ as in the proof of Lemma \ref{lem-non-vani}, we have
		\[
		H_\gamma(f) = \left(2\gamma+\sum_{j=3}^N \gamma_j\right) c \leq m
		\]
		provided $0<c<c_0$ with some $c_0=c_0(m)>0$. The claim $I^m_\gamma(c) < \omega^0_\gamma c$ follows directly from the above notice and \eqref{est-I-c}.
	\end{proof}
	
	\begin{lemma} \label{lem-key-est}
		Let $N\geq 2$, $1+\frac{4}{N}<p<1+\frac{4}{N-2}$, $V$ be as in \eqref{defi-poten-V} satisfying \eqref{axia-pote}, and $\Omega=\gamma$. Then for $m>0$ fixed, there exists $c_0=c_0(m)>0$ sufficiently small such that for all $0<c<c_0$, we have
		\begin{align}
		S(c) \cap B_\gamma(m) &\ne \emptyset, \label{non-empty} \\
		\inf \left\{E_\gamma(f) \ : \ f \in S(c) \cap B_\gamma(m/4)\right\} &< \inf \left\{ E_\gamma(f) \ : \ f \in S(c) \cap \left(B_\gamma(m) \backslash B_\gamma(m/2)\right) \right\}. \label{est-inf}
		\end{align}
	\end{lemma}
	
	\begin{proof}
		(1) Let $m>0$. We take $\varphi \in C^\infty_0(\R^N)$ be radially symmetric satisfying $\|\varphi\|_{L^2}=1$. Denote $c_0=c_0(m) = \frac{m}{C}$ with $C:= H_\gamma(\varphi)$. Set $f(x)= \sqrt{c} \varphi(x)$. We have $M(f) = c$ and $H_\gamma(f) = c H_\gamma(\varphi)< m$ for all $0<c<c_0$. This shows that $f \in S(c) \cap B_\gamma(m)$.
		
		(2) To show \eqref{est-inf}, we first note that $S(c) \cap \left(B_\gamma(m) \backslash B_\gamma(m/2)\right) \ne \emptyset$ for $c>0$ sufficiently small. In fact, let $\varphi$ be as above. Denote $f_\lambda(x):= \sqrt{c} \lambda^{\frac{N}{2}} \varphi(\lambda x)$ with $\lambda>0$ to be chosen shortly. We have $M(f_\lambda) = c$ for all $\lambda>0$. As $f_\lambda$ is radially symmetric, we have
		\[
		H_\gamma(f_\lambda) = \|\nabla f_\lambda \|^2_{L^2} + 2 \Vc(f_\lambda) =c \left( \lambda^2 \|\nabla \varphi\|^2_{L^2} + 2\lambda^{-2} \Vc(\varphi)\right).
		\]
		We will show that for each $m>0$, by reducing $c_0=c_0(m)>0$ if necessary, there exists $\lambda_0>0$ such that $H_\gamma(f_{\lambda_0}) = \frac{3m}{4}$, hence $f_{\lambda_0} \in S(c) \cap \left(B_\gamma(m) \backslash B_\gamma(m/2)\right)$. The above equality is equivalent to 
		\begin{align} \label{choice-c}
		\lambda^2 \|\nabla \varphi\|^2_{L^2} + 2\lambda^{-2} \Vc(\varphi) = \frac{3m}{4c}. 
		\end{align}
		The left hand side takes values on $\left[2\sqrt{2}\|\nabla \varphi\|_{L^2} \sqrt{\Vc(\varphi)}, \infty\right)$. Thus by taking $c_0=c_0(m)>0$ so that $\frac{3m}{4c_0} \geq 2\sqrt{2}\|\nabla \varphi\|_{L^2} \sqrt{\Vc(\varphi)}$, there exists $\lambda_0>0$ such that \eqref{choice-c} holds. 
		
		To prove \eqref{est-inf}, we observe from the magnetic Gagliardo-Nirenberg inequality that
		\begin{align*}
		E_\gamma(f) &\geq \frac{1}{2} H_\gamma(f) - B \|(\nabla-iA) f\|^{\frac{N(p-1)}{2}}_{L^2} \|f\|^{\frac{4-(N-2)(p-1)}{2}}_{L^2} \\
		&\geq \frac{1}{2} H_\gamma(f) - B \left( H_\gamma(f)\right)^{\frac{N(p-1)}{4}} (M(f))^{\frac{4-(N-2)(p-1)}{4}}, \quad \forall f \in \Sigma_\gamma
		\end{align*}
		with some constant $B>0$. It follows that
		\begin{align} \label{est-inf-proof}
		g_c(H_\gamma(f)) \leq E_\gamma(f) \leq h_c(H_\gamma(f)), \quad \forall f \in S(c),
		\end{align}
		where 
		\[
		g_c(\lambda):= \frac{1}{2} \lambda - B c^{\frac{4-(N-2)(p-1)}{4}} \lambda^{\frac{N(p-1)}{4}}, \quad h_c(\lambda):= \frac{1}{2} \lambda.
		\]
		From \eqref{est-inf-proof}, we see that \eqref{est-inf} is proved provided that there exists $c_0=c_0(m)>0$ sufficiently small such that for each $0<c<c_0$,
		\begin{align} \label{est-gh-c}
		h_c(m/4) < \inf_{\lambda \in (m/2,m)} g_c(\lambda).
		\end{align}
		Observe that
		\[
		g_c(\lambda)=\frac{1}{2}\lambda \left(1-Bc^{\frac{4-(N-2)(p-1)}{4}} \lambda^{\frac{N(p-1)-4}{4}}\right) >\frac{1}{3}\lambda
		\]
		for $\lambda \in (0,m)$ and $0<c<c_0$ with $c_0=c_0(m)>0$ sufficiently small. We infer that
		\[
		\inf_{\lambda \in (m/2,m)} g_c(\lambda) \geq \frac{m}{6} > \frac{m}{8} = h_c(m/4)
		\]
		which proves \eqref{est-gh-c}. The proof is complete.
	\end{proof}
	
	\begin{lemma} \label{lem-gwp-mass-sup}
		Let $N\geq 2$, $V$ be as in \eqref{defi-poten-V} satisfying \eqref{axia-pote}, and $\Omega=\gamma$. Let $m>0$ and $u_0 \in \Sigma_\gamma$ be such that
		\[
		H_\gamma(u_0) \leq m.
		\]
		Then there exists $c_0=c_0(m)>0$ sufficiently small such that for all $0<c<c_0$, if $M(u_0)=c$, then the corresponding solution to \eqref{RNLS} exists globally in time, i.e., $T_*=T^*=\infty$.
	\end{lemma}
	
	\begin{proof}
		Let $u:(-T_*,T^*)\times \R^N \rightarrow \C$ be the corresponding solution to \eqref{RNLS}, hence to \eqref{RNLS-mag}. By \eqref{omega-gamma-0} and \eqref{omega-gamma}, we have
		\[
		M(u_0) \leq \frac{1}{2 \omega^0_\gamma} H_\gamma(u_0) \leq \frac{m}{2\omega^0_\gamma}.
		\]
		From this and the magnetic Gagliardo-Nirenberg inequality \eqref{mag-GN-ineq}, we have
		\begin{align}
		|E_\gamma(u_0)| &\leq \frac{1}{2} H_\gamma(u_0) + B \|(\nabla-iA)u_0\|^{\frac{N(p-1)}{2}}_{L^2} \|u_0\|^{\frac{4-(N-2)(p-1)}{2}}_{L^2} \nonumber \\
		&\leq \frac{1}{2} H_\gamma(u_0) + B (H_\gamma(u_0))^{\frac{N(p-1)}{4}} (M(u_0))^{\frac{4-(N-2)(p-1)}{4}} \leq C(m) \label{est-E-u0}
		\end{align}
		for some constant $C(m)>0$ depending on $m$. On the other hand, by the conservation of mass and energy, we have for all $t\in(-T_*,T^*)$,
		\begin{align*}
		H_\gamma(u(t)) &= 2 E_\gamma(u(t)) + \frac{4}{p+1} \|u(t)\|^{p+1}_{L^{p+1}} \\
		&\leq 2E_\gamma(u(t)) + B  (H_\gamma(u(t)))^{\frac{N(p-1)}{4}}(M(u(t)))^{\frac{4-(N-2)(p-1)}{4}} \\
		&\leq 2|E(u_0)| + B  (H_\gamma(u(t)))^{\frac{N(p-1)}{4}}(M(u_0))^{\frac{4-(N-2)(p-1)}{4}}.
		\end{align*}
		Here the constant $B$ may change from line to line. In particular, we have
		\begin{align} \label{est-H-gamma}
		H_\gamma(u(t)) \leq a + b(H_\gamma(u(t))^{\frac{N(p-1)}{4}}, \quad \forall t\in (-T_*,T^*),
		\end{align}
		where
		\[
		a:= 2 |E_\gamma(u_0)| + \frac{1}{2} H_\gamma(u_0), \quad b:= B(M(u_0))^{\frac{4-(N-2)(p-1)}{4}}.
		\]
		We next claim that for each $m>0$, there exists $c_0=c_0(m)>0$ sufficiently small such that for all $0<c<c_0$, if $M(u_0) =c$, then
		\[
		H_\gamma(u(t)) \leq 2a \quad \forall t\in(-T_*,T^*).
		\]
		This together with the blow-up alternative \eqref{blow-alte-gamma} imply $T_*=T^*=\infty$.
		
		It remains to prove the claim. Assume by contradiction that it is not true. As $H_\gamma(u_0) \leq 2a$, the continuity of the solution maps ensures the existence of $t_0 \in (-T_*,T^*)$ such that $H_\gamma(u(t_0)) =2a$. Inserting into \eqref{est-H-gamma}, we get
		\[
		2a \leq a + b(2a)^{\frac{N(p-1)}{4}} \Longleftrightarrow b \geq \frac{1}{2^{\frac{N(p-1)}{4}} a^{\frac{N(p-1)-4}{4}}}.
		\] 
		From \eqref{est-E-u0}, we see that $a$ is bounded from above by a constant depending on $m$. Thus $b$ is bounded from below by some constant depending on $m$, i.e., $b\geq b_0(m)$. However, by taking $c>0$ sufficiently small so that $Bc^{\frac{4-(N-2)(p-1)}{4}} <b_0(m)$, we get a contradiction. The proof is complete.
	\end{proof}

	We are now able to give the proof of Theorem \ref{theo-mass-sup}.
	
	\begin{proof}[Proof of Theorem \ref{theo-mass-sup}]
		The proof is divided into several steps.
		
		{\bf Step 1.} From \eqref{non-empty}, we see that $I^m_\gamma(c)$ is well-defined for $c>0$ sufficiently small. Let $(f_n)_{n\geq 1}$ be a minimizing sequence for $I^m_\gamma(c)$. It follows that $(f_n)_{n\geq 1}$ is a bounded sequence in $\Sigma_\gamma$ as $(f_n)_{n\geq 1} \subset S(c) \cap B_\gamma(m)$. By Lemma \ref{lem-non-vani-sup}, there exists a subsequence still denoted by $(f_n)_{n\geq 1}$ such that 
		\[
		\inf_{n\geq 1} \|f_n\|_{L^{p+1}} \geq C>0.
		\]
		By Lemma \ref{lem-weak-conv}, there exist $\phi \in \Sigma_\gamma\backslash \{0\}$ and a sequence $(y_n)_{n\geq 1} \subset \R^N$ satisfying $y_n=(y^1_n,y^2_n, 0,\cdots, 0)$ such that up to a subsequence,
		\[
		g_n(x):= e^{iA(y_n)\cdot x} f_n(x+y_n) \rightharpoonup \phi \text{ weakly in } \Sigma_\gamma.
		\]
		By the weak convergence, we have
		\[
		0<M(\phi)\leq \liminf_{n\rightarrow \infty} M(g_n) = \liminf_{n\rightarrow \infty} M(f_n)=c
		\]
		and
		\[
		H_\gamma(\phi) \leq \liminf_{n\rightarrow \infty} H_\gamma(g_n) = \liminf_{n\rightarrow \infty} H_\gamma(f_n) \leq m.
		\]
		Moreover, by the same argument as in the proof of Theorem \ref{theo-mass-sub}, we prove that $M(\phi)=c$ or $\phi \in S(c) \cap B_\gamma(m)$. We also have $E_\gamma(\phi) = I^m_\gamma(c)$ or $\phi$ is a minimizer for $I^m_\gamma(c)$. In addition, $g_n\rightarrow \phi$ strongly in $\Sigma_\gamma$. 
		
		{\bf Step 2.} We next prove that $\Mcal_\gamma^m(c) \subset B_\gamma(m/2)$. Indeed, let $\phi \in \Mcal^m_\gamma(c)$ and assume by contradiction that $\phi \notin B_\gamma(m/2)$. By \eqref{est-inf}, we have
		\begin{align*}
		I^m_\gamma(c) &\leq \inf \left\{E_\gamma(f) \ : \ f \in S(c) \cap B_\gamma(m/4)\right\} \\
		&< \inf\left\{ E_\gamma(f) \ : \ f \in S(c) \cap \left( B_\gamma(m) \backslash B_\gamma(m/2)\right)\right\} \\
		&\leq E_\gamma(\phi) = I^m_\gamma(c)
		\end{align*}
		which is a contradiction. 
		
		As $\phi$ does not belong to the boundary of $B_\gamma(m)$, there exists a Lagrange multiplier $\omega \in \R$ such that $S'_{\gamma,\omega}(\phi)[\varphi]=0$ for all $\varphi \in C^\infty_0(\R^N)$, where $S_{\gamma, \omega}(f):= E_\gamma(f) + \omega M(f)$. It shows that $\phi$ is a weak solution to 
		\begin{align} \label{equ-phi}
		-\frac{1}{2}(\nabla-iA)^2 \phi + V_\gamma \phi - |\phi|^{p-1}\phi + \omega \phi =0
		\end{align}
		or $\phi$ is a solution to \eqref{ell-equ-intro} in the weak sense. In particular, $u(t,x) = e^{i\omega t} \phi(x)$ is a solution to \eqref{RNLS}. We also have from \eqref{equ-phi} that
		\begin{align*}
		\omega M(\phi) &= - \frac{1}{2} \|(\nabla-iA) \phi\|^2_{L^2} - \int_{\R^N} V_\gamma(x)|\phi(x)|^2 dx + \|\phi\|^{p+1}_{L^{p+1}} \\
		&= - E_\gamma(\phi) + \frac{p-1}{p+1} \|\phi\|^{p+1}_{L^{p+1}}>-E_\gamma(\phi).
		\end{align*}
		This together with Lemma \ref{lem-non-vani-sup} yield
		\begin{align} \label{est-omega-proof-1}
		\omega > \frac{-E_\gamma(\phi)}{M(\phi)} = -\frac{I^m_\gamma(c)}{c} > -\omega^0_\gamma.
		\end{align}
		On the other hand, by the magnetic Gagliardo-Nirenberg inequality as in \eqref{est-E-u0}, we have
		\begin{align*}
		\omega M(\phi) &\leq - \frac{1}{2} H_\gamma(\phi) + B (H_\gamma(\phi))^{\frac{N(p-1)}{4}} (M(\phi))^{\frac{4-(N-2)(p-1)}{4}} \\
		&\leq - \frac{1}{2} H_\gamma(\phi) \left(1- 2B (H_\gamma(\phi))^{\frac{N(p-1)-4}{4}} (M(\phi))^{\frac{4-(N-2)(p-1)}{4}}\right) 
		\end{align*}
		for some constant $B>0$. As $\phi \in S(c) \cap B_\gamma(m/2)$, we get
		\[
		\omega M(\phi) \leq -H_\gamma(\phi) \left(1-B  m^{\frac{N(p-1)-4}{4}} c^{\frac{4-(N-2)(p-1)}{4}} \right),
		\]
		where the constant $B>0$ may change from line to line. Reducing $c>0$ if necessary, it follows from \eqref{omega-gamma-0} that
		\begin{align} \label{est-omega-proof-2}
		\omega \leq -\omega^0_\gamma \left(1-B m^{\frac{N(p-1)-4}{4}} c^{\frac{4-(N-2)(p-1)}{4}}\right).
		\end{align}
		Collecting \eqref{est-omega-proof-1} and \eqref{est-omega-proof-2}, we prove \eqref{est-omega}. 
		
		Let us now prove \eqref{est-sup}. Let $\phi \in \Mcal^m_\gamma(c)$. As $M(\phi) = c$ and $H_\gamma(\phi) \leq m$, we have from the magnetic Gagliardo-Nirenberg inequality and Lemma \ref{lem-non-vani-sup} that
		\begin{align*}
		H_\gamma(\phi) &= 2E_\gamma(\phi) +\frac{4}{p+1}\|\phi\|^{p+1}_{L^{p+1}} \\
		&\leq 2I^m_\gamma(c) + B m^{\frac{N(p-1)}{4}} c^{\frac{4-(N-2)(p-1)}{4}} \\
		&\leq 2\omega^0_\gamma c+ B m^{\frac{N(p-1)}{4}} c^{\frac{4-(N-2)(p-1)}{4}}
		\end{align*}
		for some constant $B>0$ independent of $m$ and $c$. It follows that
		\[
		\|\phi\|^2_{\Sigma_\gamma} \simeq H_\gamma(\phi) + M(\phi) \leq (2\omega^0_\gamma+1) c + B m^{\frac{N(p-1)}{4}} c^{\frac{4-(N-2)(p-1)}{4}}
		\]
		which shows \eqref{est-sup}. 
		
		{\bf Step 3.} We will prove that $\Mcal^m_\gamma(c)$ is orbitally stable under the flow of \eqref{RNLS}. As in the proof of Theorem \ref{theo-mass-sub}, we argue by contradiction. Suppose that $\Mcal^m_\gamma(c)$ is not orbitally stable. By definition, there exist $\varepsilon_0>0, \phi_0 \in \Mcal^m_\gamma(c)$, a sequence $(u_{0,n})_n \subset \Sigma_\gamma$ satisfying
		\begin{align}  \label{orbi-proof-1}
		\lim_{n\rightarrow \infty} \|u_{0,n} - \phi_0\|_{\Sigma_\gamma} = 0,
		\end{align}
		and a sequence of time $(t_n)_{n\geq 1} \subset \R$ such that
		\begin{align} \label{orbi-proof-2}
		\inf_{\phi \in \Mcal^m_\gamma(c)} \inf_{y \in \Theta} \|u_n(t_n) - \phi\|_{\Sigma_\gamma} \geq \varepsilon_0,
		\end{align}
		where $u_n$ is the solution to \eqref{RNLS} with initial data $u_n(0)= u_{0,n}$. Note that the solutions exist globally in time by Lemma \ref{lem-gwp-mass-sup}.
		
		Since $\phi_0 \in \Mcal^m_\gamma(c)$, we have $E_\gamma(\phi_0) = I^m_\gamma(c)$. By \eqref{orbi-proof-1} and Sobolev embedding, we have
		\[
		M(u_{0,n}) \rightarrow M(\phi_0)=c, \quad H_\gamma(u_{0,n}) \rightarrow H_\gamma(\phi_0) \leq m, \quad E_\gamma(u_{0,n}) \rightarrow E_\gamma(\phi_0) = I^m_\Omega(c)
		\]
		as $n\rightarrow \infty$. By conservation laws of mass and energy, we have
		\[
		M(u_n(t_n)) \rightarrow c, \quad E_\gamma(u_n(t_n)) \rightarrow I^m_\gamma(c)
		\]
		as $n\rightarrow \infty$. We next claim that (up to a subsequence) $H_\gamma(u_n(t_n))\leq m$ for all $n\geq 1$. Suppose that there exists $K\geq 1$ such that $H_\gamma(u_n(t_n)) >m$ for every $n\geq K$. By continuity, there exists $t_n^*$ such that $H_\gamma(u_n(t^*_n)) =m$. Since 
		\[
		M(u_n(t^*_n)) \rightarrow c, \quad H_\gamma(u_n(t^*_n))=m, \quad E_\gamma(u_n(t^*_n)) \rightarrow I^m_\gamma(c)
		\]
		as $n\rightarrow \infty$, we see that $u_n(t^*_n)$ is a minimizing sequence for $I^m_\gamma(c)$. By Step 1, there exists $\phi \in S(c)\cap B_\gamma(m)$ such that $u_n(t^*_n) \rightarrow \phi$ strongly in $\Sigma_\gamma$ and $E_\gamma(\phi) = I^m_\gamma(c)$. This is not possible since minimizers for $I^m_\gamma(c)$ does not belong to the boundary of $B_\gamma(m)$. Thus there exists a subsequence $(t_{n_k})_{k\geq 1}$ such that $H_\gamma(u_{n_k}(t_{n_k}))\leq m$ for all $k\geq 1$. This shows that $(u_{n_k}(t_{n_k}))_{k\geq 1}$ is a minimizing sequence for $I^m_\gamma(c)$. Again, by Step 1, there exist $\phi \in \Mcal^m_\gamma(c)$ and a sequence $(y_k)_{k\geq 1} \subset \Theta$ such that
		\[
		\|e^{iA(y_k)\cdot \cdotb}u_{n_k}(t_{n_k}, \cdotb+y_k) - \phi\|_{\Sigma_\gamma} \rightarrow 0
		\]
		as $k\rightarrow \infty$. This contradicts \eqref{orbi-proof-2}, and the proof is complete.	
	\end{proof}
	
	\begin{proof}[Proof of Theorem \ref{theo-mass-sup-prop}]
		Let $\phi \in \Mcal^m_\gamma(c)$. As $\phi$ does not belong to the boundary of $B_\gamma(m)$, we have $\left.E'_\gamma\right|_{S(c)}(\phi)=0$. 
		
		Assume by contradiction that there exists $f \in S(c)$ with $\left.E'_\gamma\right|_{S(c)}=0$ such that $E_\gamma(f) < E_\gamma(\phi)=I^m_\gamma(c)$. Since $\left.E'_\gamma\right|_{S(c)}=0$, there exists a Lagrange multiplier $\omega \in \R$ such that $f$ is a solution to 
		\[
		-\frac{1}{2}(\nabla-iA)^2 f + V_\gamma f - |f|^{p-1}f + \omega f =0.
		\]
		Multiplying both sides of the above equation with $\overline{f}$ and integrating over $\R^N$, we have
		\[
		\frac{1}{2} H_\gamma(f) + \omega M(f) - \|f\|^{p+1}_{L^{p+1}}=0.
		\]
		From this and the fact that $I^m_\gamma(c)< \omega^0_\gamma c$, we have
		\[
		\frac{p-1}{p+1} \left(\frac{1}{2} H_\gamma(f)+\omega M(f)\right) =E_\gamma(f) + \omega M(f) < I^m_\gamma(c) + \omega c < (\omega^0_\gamma+\omega) c.
		\]
		Thus we get
		\[
		H_\gamma(f) < \frac{2(p+1)}{p-1} (\omega^0_\gamma+\omega)c - 2\omega c
		\]
		which, by \eqref{est-omega}, shows that $H_\gamma(f) \rightarrow 0$ as $c\rightarrow 0$. In particulart, for $c>0$ sufficiently small, we have $f \in S(c) \cap B_\gamma(m)$. By the definition of $I^m_\gamma(c)$, we get $I^m_\gamma(c) \leq E_\gamma(f)$ which is a contradiction. The proof is complete.
	\end{proof}
	
	\section*{Acknowledgments}
	This work was supported in part by the European Union’s Horizon 2020 Research and Innovation Programme (Grant agreement CORFRONMAT No. 758620, PI: Nicolas Rougerie). V.D.D. would like to express his deep gratitude to his wife - Uyen Cong for her encouragement and support. The authors would like to thank the reviewers for their helpful comments and suggestions. 
	

\end{document}